\documentclass[11pt,a4paper]{article}

\usepackage{comment}
\usepackage{graphicx}
\usepackage{amsmath}
\usepackage{amsfonts}
\usepackage{amssymb}
\usepackage{enumerate}
\usepackage{lscape}
\usepackage{longtable}
\usepackage{rotating}
\usepackage{color}
\usepackage{url}
\usepackage{rotating}
\usepackage{algpseudocode}
\usepackage[ruled]{algorithm}
\usepackage{pgf,tikz}
\usepackage{subfig}
\usepackage[titletoc]{appendix}
\usepackage{wrapfig}
\usepackage{adjustbox}
\usepackage{mathrsfs}
\usepackage{pgfplots}
\usetikzlibrary{arrows}
\usepackage{color}
\usepackage{jingfu}
\usepackage{bbm}
\usepackage{booktabs}
\usepackage{multirow}
\usepackage{mathtools}
\usepackage{makecell}

\usepackage[sort,numbers]{natbib}
\usepackage{eqparbox}%

\newcounter{starnumber}
\newcommand{\stars}[1]{
  \forloop{starnumber}{1}{\value{starnumber} < 4}{
    \ifthenelse{#1 < \value{starnumber}}{\ding{73}}{\ding{72}}%
  }
}

\numberwithin{equation}{section}%

\usepackage[symbol]{footmisc}%

\newtheorem{theorem}{Theorem}[section]
\newtheorem{proposition}{Proposition}[section]

\newtheorem{corollary}{Corollary}[section]

\newtheorem{lemma}[theorem]{Lemma}

\newtheorem{remark}{Remark}[section]

\newenvironment{proof}[1][Proof]{\noindent \textbf{#1.} }{\hfill$\Box$\par\medskip}

\DeclareMathOperator*{\argmin}{argmin}

\DeclareMathOperator{\interior}{int}

\DeclareMathOperator{\conv}{conv}

\newcommand{\beqn}[1]{\begin{equation}\label{#1}}
\newcommand{\eeqn}{\end{equation}}

\newcommand\shortdash{\mathop{\mbox{-}}}

\definecolor{darkgreen}{rgb}{0,0.6,0}
\definecolor{aau2}{rgb}{0.0, 0.5, 0.69}
\definecolor{aau3}{rgb}{0.0, 0.53, 0.74}
\definecolor{aau4}{rgb}{0.0, 0.48, 0.65}
\definecolor{aau5}{rgb}{0.0, 0.45, 0.73}
\definecolor{rsap}{RGB}{130, 36, 51}
\definecolor{gsap}{RGB}{112, 164, 137}

\definecolor{tud}{rgb}{0.43,0.73,0.11}
\definecolor{verde}{rgb}{0.33,0.53,0.11}

\definecolor{ttffqq}{rgb}{0.0, 0.48, 0.65} 
\definecolor{ffqqqq}{rgb}{0.0, 0.5, 0.69} 

\usetikzlibrary{arrows}

\tikzstyle{decision} = [diamond, draw, fill=blue!20,
text width=4.5em, text badly centered, node distance=3cm, inner sep=0pt]
\tikzstyle{block} = [rectangle, draw, fill=blue!20,
text centered, rounded corners, minimum height=4em]
\tikzstyle{line} = [draw, -latex']
\tikzstyle{cloud} = [draw, ellipse,fill=red!20, node distance=3cm,
minimum height=2em]
\tikzstyle{cloud2} = [draw, ellipse,fill=green!20, node distance=3cm,
minimum height=2em]

\usetikzlibrary{patterns}

\begin{document}
	
	\title{Stochastic set-valued optimization \\ and its application to robust learning}
	
	\author{
		T. Giovannelli\thanks{Department of Mechanical and Materials Engineering, University of Cincinnati, Cincinnati, OH 45221, USA ({\tt giovanto@ucmail.uc.edu}).}
		\and
		J. Tan\thanks{Department of Industrial and Systems Engineering, Lehigh University, Bethlehem, PA 18015-1582, USA ({\tt jit423@lehigh.edu}).}
		\and
		L. N. Vicente\thanks{Department of Industrial and Systems Engineering, Lehigh University, Bethlehem, PA 18015-1582, USA ({\tt lnv@lehigh.edu}).}
	}
	
	\maketitle
	
 \begin{abstract}
In this paper, we develop a stochastic set-valued optimization (SVO) framework tailored for robust machine learning. In the SVO setting, each decision variable is mapped to a set of objective values, and optimality is defined via set relations. 
We focus on SVO problems with hyperbox sets, which can be reformulated as multi-objective optimization (MOO) problems with finitely many objectives and serve as a foundation for representing or approximating more general mapped sets. Two special cases of hyperbox-valued optimization (HVO) are interval-valued (IVO) and rectangle-valued (RVO) optimization.
We construct stochastic IVO/RVO formulations that incorporate subquantiles and superquantiles into the objective functions of the MOO reformulations,
providing a new characterization for subquantiles.
These formulations provide interpretable trade-offs by capturing both lower- and upper-tail behaviors of loss distributions, thereby going beyond standard empirical risk minimization and classical robust models. To solve the resulting multi-objective problems, we adopt stochastic multi-gradient algorithms and select a Pareto knee solution. In numerical experiments, the proposed algorithms with this selection strategy exhibit improved robustness and reduced variability across test replications under distributional shift compared with empirical risk minimization, while maintaining competitive accuracy.
 \end{abstract}

\section{Introduction} 

    In set-valued optimization~(SVO), the objective is to minimize a set-valued mapping~$S$ from~$\mathbb{R}^n$ to~$\mathbb{R}^m$, 
    which assigns to each point in~$\mathbb{R}^n$ a set of points in~$\mathbb{R}^m$, rather than a single value. Two main solution concepts can be used to define optimality: the set approach and the vector approach. In the set approach, optimal solutions (referred to as minimal solutions in~SVO contexts) are determined by comparing image sets of~$S$, seeking a point~$x_* \in \mathbb{R}^n$ whose image set~$S(x_*) \subseteq \mathbb{R}^m$ is not dominated by others based on a set relation. In contrast, the vector approach focuses on comparing individual vectors within the image sets rather than the sets themselves, seeking a point~$x_* \in \mathbb{R}^n$ for which there exists a point~$y_* \in S(x_*) \subseteq \mathbb{R}^m$ that is non-dominated with respect to all points in~$\cup_{x \in \mathbb{R}^n} S(x)$. 
    
    Our paper adopts the set approach, as it aligns well with the robust learning applications we will propose. Among the set relations proposed for comparing the image sets of a set-valued mapping at different solutions, we will focus on the well-known set-less relations introduced by Kuroiwa et al.~\cite{DKuroiwa_1998}. Intuitively, when using the first orthant as an ordering cone, an image set~$S(\bar{x})$ dominates another image set~$S(\tilde{x})$, with~$\bar{x} \neq \tilde{x}$, based on the lower (upper) set-less relation if the points with the minimum (maximum) objective values in~$S(\bar{x})$ are below the corresponding points in~$S(\tilde{x})$. A feasible point~$x_*$ is referred to as a minimal solution if its image set~$S(x_*)$ is not dominated by the image set of any other feasible point. This concept is analogous to efficient solutions in vector-valued optimization~(VVO), where the objective is a vector function and the image of a feasible point under such a function is a vector in~$\Rmbb^m$. Specifically, a feasible point is called an efficient solution if its image is not dominated by the image of any other feasible point. When the ordering cone is the non-negative orthant~$\mathbb{R}^m_+$, which is the case of interest for our paper, VVO reduces to multi-objective optimization~(MOO), and efficient solutions are called Pareto optimal solutions.

    More formally, given a set-valued mapping~$S: \mathbb{R}^n \rightrightarrows \mathbb{R}^m$, a closed, convex, and pointed cone~$K \subset \mathbb{R}^m$ (where convex means~$K + K = K$ and pointed means~$K \cap -K = \{0\}$), and a partial order relation between sets with respect to~$K$, denoted by the set-less relation~$\preceq_K$ (which will be rigorously introduced in Section~\ref{sec:notation_basic_definitions}), an~SVO problem can be formulated as follows
    \begin{equation}\label{prob:svo}
    \preceq_K \shortdash \min _{x \in X} S(x),
    \end{equation}
    where~$X \subseteq \mathbb{R}^n$.
    A class of set-valued optimization problems~\eqref{prob:svo}, considered in this paper, is hyperbox-valued optimization (HVO), in which each image set~$S(x)\subset\mathbb{R}^m$ is a hyperbox specified by component-wise lower and upper bounds (see~\cite{ahmad2013interval,eichfelder2025two}). HVO provides a foundation for representing or approximating more general set-valued mappings, and two special cases are interval-valued optimization (IVO), when~$m=1$ and~$S(x)$ is an interval in~$\mathbb{R}$, and rectangle-valued optimization (RVO), when~$m=2$ and~$S(x)$ is a rectangle in~$\mathbb{R}^2$. Low-dimensional hyperboxes, such as intervals and rectangles, lead to HVO formulations that are solvable and interpretable.

    \paragraph{Robust learning motivation.} 
    Machine learning~(ML) robustness refers to a model's ability to maintain its performance despite uncertainties or adversarial conditions~\cite{JZLi_2018,JSteinhardt_2018}. In practical applications, data is often noisy, incomplete, subject to adversarial attacks, or unrepresentative of the entire population~\cite{MAGianfrancesco_etal_2018,AMKhan_etal_2014}, leading to unreliable and inconsistent results in terms of accuracy and fairness across different datasets~\cite{NMehrabi_etal_2022,HEvans_DSnead_2024,SHajian_FBonchi_CCastillo_2016,ENtoutsi_etal_2020}. In today's data-driven world, robust~ML techniques can address such challenges by ensuring that models remain resilient in worst-case scenarios and enhancing their generalizability, allowing them to perform effectively on unseen data. 
    Therefore, through robust learning, we can build~ML models that are not only consistently accurate and fair but also resilient to unexpected future circumstances.
    
    Optimization methods are instrumental components to achieve robust~ML models and have been incorporated into~ML application problems to achieve various forms of robustness. 
    Adversarial robustness ensures a model withstands data perturbations designed to mislead it, and it is typically addressed through min-max or bilevel optimization, where the max/lower-level problem is posed on the variables that perturb the data in a worst-case fashion, while the min/upper-level problem minimizes the training loss function on the model parameters~\cite{CSzegedy_etal_2013}. Robustness to imbalanced data adjusts weights in bilevel formulations to penalize misclassification of underrepresented classes~\cite{MSOzdayi_MKantarcioglu_RIyer_2021,YRoh_etal_2020,MMKamani_etal_2020}. Robustness to overfitting prevents the model from excessively fitting training data, and it is achieved through~L1/L2 regularization, cross-validation/hyperparameter tuning (which can be framed as a bilevel problem), or distributionally robust optimization (which can be framed as a min-max problem)~\cite{HNamkoong_JCDuchi_2016,Dlevy_etal_2020,JDuchi_HNamkoong_2021}, where the goal is to ensure that a model performs well across a range of possible data distributions.

    Hence, the main optimization approaches for robust~ML require tuning parameters, optimizing at different levels, or considering different data distributions.  
    Additionally, traditional robust ML techniques, including bilevel, min-max, and regularization methods, fall short in capturing the full complexity of uncertainty in real-world data by reducing variability to a single measure: the empirical risk. 
    
    \paragraph{A brief SVO literature review.} Research in~SVO began in the late~1990s and early~2000s with Kuroiwa's work~\cite{DKuroiwa_1998,DKuroiwa_2001}, which introduced set relations to compare sets and establish preferences among them. 
    A comprehensive resource on~SVO is the book by Khan et al.~\cite{AMKhan_etal_2014}, which covers foundational concepts for set-valued mappings derived from variational analysis~\cite{RTRockafellar_RJBWets_1998}, and also discusses optimality conditions, duality theory, and numerical methods for~SVO. Another popular introduction to the topic is the review by Hamel et al.~\cite{AHHamel_etal_2015}.
     
    In~2015, Jahn made a fundamental contribution to the field of~SVO by proposing vectorization, which transforms an~SVO problem into a~VVO problem by reducing the comparison of two sets to the pointwise comparison of functions in an infinite-dimensional vector space~\cite{JJahn_2015}. Vectorization for~SVO plays the same role that scalarization plays for~VVO, where a~VVO problem is reduced to a real-valued optimization problem by weighting the components of the vector function into a single objective (weighted-sum method) or minimizing one objective subject to the others up to certain thresholds~($\varepsilon$-constrained method)~\cite{MEhrgott_2005}. In~\cite{GEichfelder_SRocktaschel_2023}, Eichfelder and Rockt\"aschel provided error bounds on the approximation used to discretize the infinite-dimensional vector space resulting from vectorization. In~\cite{GEichfelder_TGerlach_2019}, Eichfelder and Gerlach identified classes of~SVO that can be formulated as~MOO problems in a finite-dimensional objective space, including box-valued and ball-valued optimization.

    To the best of our knowledge and according to the book by~Khan et~al.~\cite{AAKhan_etal_2015} and survey~\cite{AHHamel_etal_2015}, no stochastic formulations for~SVO have been proposed so far. The deterministic algorithms in~\cite{AAKhan_etal_2015} include a Newton method and an algorithm for solving polyhedral convex~SVO problems. In~2015, Jahn proposed an algorithm that performs enumerative pairwise comparisons between sets~\cite{JJahn_DFO_2015}. In~2021, a deterministic gradient descent method for~SVO problems with a set-valued mapping of finite cardinality was developed by~Bouza et~al.~\cite{GBouza_2021}.

\paragraph{Our contributions.} 
Motivated by the lack of large-scale, application-driven methodologies for set-valued optimization, this paper develops a stochastic SVO framework in which the stochasticity arises from an underlying data distribution. For this purpose, we consider hyperbox-valued optimization (HVO) models, in particular interval-valued optimization (IVO) and rectangle-valued optimization (RVO). HVO is known to admit an equivalent reformulation as a standard vector-valued optimization (VVO) problem with a finite number of vector components. The reformulated VVO problem can be viewed as a multi-objective optimization (MOO) problem with a finite number of objective functions (2~for IVO and 4~for RVO), obtained by treating each component of the objective vector as a separate objective.

Our stochastic set-valued objectives for IVO and RVO are defined through conditional expectation for super and subquantiles and are accessed in practice through random samples or mini-batches. A superquantile is also known as Conditional Value-at-Risk or expected shortfall and enjoys the convex characterization of Rockafellar and Uryasev~\cite{RTRockafellar_SUryasev_2000}.
One contribution of this paper is the introduction and proof of a Rockafellar--Uryasev-type characterization for subquantiles, mirroring the classical representation of superquantiles. 

In robust learning, IVO is appealing because it characterizes model performance via lower and upper assessments of the loss, rather than relying on a single average score, thereby retaining more information and potentially improving robustness on testing data. RVO extends this idea to settings where robustness must balance multiple criteria simultaneously, such as a predictive accuracy and a fairness metric. Building on the corresponding MOO reformulation, we design robust learning objectives from lower- and upper-tail risk measures of the loss via sub-/superquantile-type functionals. The resulting IVO and RVO formulations capture the relationship between uncertainty and solution quality, offering an alternative way to interpret loss-tail behavior and robustness trade-offs compared to traditional robust ML techniques such as bilevel, min--max, and regularization methods. We train the resulting models using stochastic mini-batch multi-gradient descent and then extract a single MOO solution via a curvature-based ``knee'' selection on the Pareto front. Across experiments under distribution shift, the proposed approach produces solutions that are more robust than single-objective empirical risk minimization while maintaining competitive predictive accuracy. 

The remainder of the paper is organized as follows. Section~\ref{sec:background} introduces the notation and background on partial orders for vectors and sets, set-valued mappings, and the set-less relations used to define optimality, along with a brief review of vectorization in set-valued optimization. Section~\ref{sec: hyperbox} focuses on hyperbox-valued optimization and shows how the special hyperbox structure yields an equivalent MOO reformulation with a finite number of objectives. Section~\ref{sec:stoch_bvo} develops stochastic HVO formulations based on subquantiles and superquantiles, including the resulting stochastic interval- and rectangle-valued models as two special cases. (The Rockafellar--Uryasev-type characterization for subquantiles is established in an appendix.)
Section~\ref{sec:robust_learning_bvo} applies these models to robust learning, describes the corresponding training and evaluation procedures, explains how we select a representative Pareto optimal solution, and reports numerical experiments under a distributional shift. Section~\ref{sec:conclusion} concludes and discusses directions for future work.
    
\section{Background}\label{sec:background}

\subsection{Notation and basic definitions}\label{sec:notation_basic_definitions}

Throughout the paper, we denote the non-negative orthant of~$\mathbb{R}^m$ as~$\mathbb{R}^m_{+}$.

\subsubsection{Partial order relations between vectors}
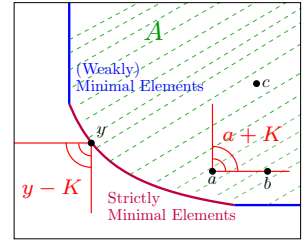
\begin{wrapfigure}[]{R}{0.30\textwidth}
  \centering
  \begin{tikzpicture}[scale=0.55]
\centering
\begin{axis}
[xmin=0, xmax=2.6,
ymin=-0.5, ymax=3,
axis on top=true,
xtick=\empty,
ytick=\empty,
domain=0.1:3,]

\addplot [mark=none, draw=blue, ultra thick, solid, domain=0:0.5] coordinates {(0.5,1.5) (0.5,3)};

\addplot [mark=none, draw=purple, ultra thick, solid, domain=0.5:2] {1/x - 0.5}; 

\node[anchor=south west] at (axis cs:0.8, -0.1) {\textcolor{purple}{Strictly}};
\node[anchor=south west] at (axis cs:0.8, -0.3) {\textcolor{purple}{Minimal Elements}};
\node[anchor=south west] at (axis cs:0.51, 1.8) {\textcolor{blue}{(Weakly)}};
\node[anchor=south west] at (axis cs:0.51, 1.6) {\textcolor{blue}{Minimal Elements}};

\addplot [mark=none, draw=blue, ultra thick, solid, domain=2:3] coordinates {(2,0) (3,0)};

\addplot[style={solid}, draw=darkgreen, dashed, ultra thin, opacity=0.6] coordinates {(2.3, 0) (8, 5.4)};
\addplot[style={solid}, draw=darkgreen, dashed, ultra thin, opacity=0.6] coordinates {(2, 0.1) (8, 5.6)};
\addplot[style={solid}, draw=darkgreen, dashed, ultra thin, opacity=0.6] coordinates {(1.85, 0.15) (8, 6)};
\addplot[style={solid}, draw=darkgreen, dashed, ultra thin, opacity=0.6] coordinates {(1.65, 0.2) (8, 6.2)};
\addplot[style={solid}, draw=darkgreen, dashed, ultra thin, opacity=0.6] coordinates {(1.48, 0.25) (8, 6.6)};
\addplot[style={solid}, draw=darkgreen, dashed, ultra thin, opacity=0.8] coordinates {(1.3, 0.35) (8, 7)};
\addplot[style={solid}, draw=darkgreen, dashed, ultra thin, opacity=0.8] coordinates {(1.05, 0.45) (8, 7.2)};
\addplot[style={solid}, draw=darkgreen, dashed, ultra thin, opacity=0.8] coordinates {(0.87, 0.65) (8, 7.8)};
\addplot[style={solid}, draw=darkgreen, dashed, ultra thin, opacity=0.8] coordinates {(0.75, 0.85) (8, 8.1)};
\addplot[style={solid}, draw=darkgreen, dashed, ultra thin, opacity=0.8] coordinates {(0.68, 1) (8, 8.3)};
\addplot[style={solid}, draw=darkgreen, dashed, ultra thin, opacity=0.8] coordinates {(0.65, 1.27) (8, 8.27)};
\addplot[style={solid}, draw=darkgreen, dashed, ultra thin, opacity=0.8] coordinates {(0.6, 1.5) (8, 8.5)};
\addplot[style={solid}, draw=darkgreen, dashed, ultra thin, opacity=0.8] coordinates {(0.53, 1.75) (8, 8.75)};
\addplot[style={solid}, draw=darkgreen, dashed, ultra thin, opacity=0.8] coordinates {(0.5, 2.0) (8, 9)};
\addplot[style={solid}, draw=darkgreen, dashed, ultra thin, opacity=0.8] coordinates {(0.5, 2.25) (8, 9)};
\addplot[style={solid}, draw=darkgreen, dashed, ultra thin, opacity=0.8] coordinates {(0.5, 2.45) (8, 9)};

\addplot[only marks, mark=*, color=black, mark options={scale=1.0}] coordinates {(2.2, 1.8) (1.8, 0.5) (2.3, 0.5) (0.7, 0.92)};
\path (axis cs:2.2, 1.8) node[right] {\large $c$};
\path (axis cs:1.8, 0.5) node[below] {\large $a$};
\path (axis cs:2.3, 0.5) node[below] {\large $b$};
\path (axis cs:0.7, 0.92) node[above right] {\large $y$};

\draw[red, thick] (axis cs:1.8, 0.5) -- (axis cs:1.8, 1.5); 
\draw[red, thick] (axis cs:1.8, 0.5) -- (axis cs:2.5, 0.5); 

\draw [red, thick] (axis cs:1.92, 0.5) coordinate (angle) arc[start angle=0, end angle=90, radius=0.3cm];
\draw [red, thick] (axis cs:2.03, 0.5) coordinate (angle) arc[start angle=0, end angle=90, radius=0.6cm];



\draw[red, thick] (axis cs:0.7, 0.92) -- (axis cs:0.7, -0.12); 
\draw[red, thick] (axis cs:0.7, 0.92) -- (axis cs:-0.3, 0.92); 

\draw [red, thick] (axis cs:0.7, 0.73) coordinate (angle) arc[start angle=-90, end angle=-180, radius=0.3cm];
\draw [red, thick] (axis cs:0.7, 0.56) coordinate (angle) arc[start angle=-90, end angle=-180, radius=0.6cm];

\end{axis}


\node[text width=1cm,font=\small] at (4, 5) { \textcolor{darkgreen}{$A$}};

\node[text width= 2cm, font=\scriptsize] at (2.0, 1.2) { \textcolor{red}{$y - K$}};

\node[text width= 2cm, font=\scriptsize] at (6.8, 2.5) { \textcolor{red}{$a + K$}};

\end{tikzpicture}
  \caption{Partial order relations in~$\mathbb{R}^2$ for $K = \mathbb{R}^2_+$.}\label{fig:VVO}
\end{wrapfigure}
Let~$a$, $b$, and~$c$ be vectors in~$A \subseteq \mathbb{R}^m$ (see Fig.~\ref{fig:VVO}). A closed, convex, and pointed cone~$K \subset \mathbb{R}^m$ induces a partial order among such vectors. Specifically, we denote~$a \le_K b$ if~$b - a \in K$, which means that~$a$ \textit{(weakly) dominates}~$b$ with respect to (w.r.t.)~$K$.
If~$K$ has non-empty interior~$\interior(K)$, $K$ also induces a strict partial order, in the sense that~$a <_K c$ if~$c - a \in \interior(K)$, which means that~$a$ \textit{strictly dominates}~$c$ w.r.t.~$K$. Note that when~$K = \mathbb{R}^m_{+}$, which denotes the non-negative orthant of~$\mathbb{R}^m$, $a \le_K b$ implies~$a_i \le b_i$ and~$a <_K c$ implies~$a_i < c_i$, for all~$i \in \{1,\ldots,m\}$.

We say that a point~$y \in A$ is a \textit{strictly minimal element} of~$A$ w.r.t~$K$ if one of the following equivalent relations is satisfied
\begin{gather}
(y - K) \cap A = \{y\} \nonumber\\
\Updownarrow \nonumber\\
\not\exists z \in A \text{ such that } z \le_K y \text{ and } z \ne y \label{eq:v50}\\
\Updownarrow \nonumber\\
\forall z \in A \text{ such that } z \le_K y, \text{ we have } z = y \nonumber\\
\Updownarrow \nonumber\\
\forall z \in A \text{ such that } z \le_K y, \text{ we have } y \le_K z.\label{eq:v10}
\end{gather}
 
 Similarly, a point~$y \in A$ is a \textit{(weakly) minimal element} of~$A$ w.r.t.~$K$ if one of the following equivalent relations is satisfied
\begin{gather}
(y - \interior(K)) \cap A = \varnothing \nonumber\\
\Updownarrow \nonumber\\
\not\exists z \in A \text{ such that } z <_K y. \label{eq:v20}
\end{gather}

\subsubsection{Partial order relation between sets}
     Let~$\mathcal{F}$ be a family of subsets of $\mathbb{R}^m$ and let us consider non-empty sets~$A \in \mathcal{F}$ and~$B \in \mathcal{F}$. A closed, convex, and pointed cone~$K$ induces a partial order between such sets. 
     Among the most popular set relations, one can include (see Fig.~\ref{fig:SVO}~(a)):
    \begin{equation*}
    \begin{split}
    \text{upper set-less:} \quad & A \preceq^{u}_K B \\ & \forall a \in A, \exists b \in B \text{ such that } a \le_K b \iff A \subset B - K, \\\vspace{0.2cm}
    \text{lower set-less:} \quad & A \preceq^{\ell}_K B \\ & \forall b \in B, \exists a \in A \text{ such that } a \le_K b \iff B \subset A + K, \\
    \text{set-less:} \quad & A \preceq_K B \iff A \preceq^{u}_K B \text{ and } A \preceq^{\ell}_K B.\\
    \end{split}
    \end{equation*}

    By replacing~$\preceq_K$ with~$\prec_K$, one can consider corresponding strict set relations, where~$K$ is replaced by its interior~$\interior(K)$. The corresponding strict set relations are defined as follows:
    \begin{equation*}
    \begin{split}
    \text{upper set-less:} \quad & A \prec^{u}_K B \\ & \forall a \in A, \exists b \in B \text{ such that } a <_K b \iff A \subset B - \interior(K), \\\vspace{0.2cm}
    \text{lower set-less:} \quad & A \prec^{\ell}_K B \\ & \forall b \in B, \exists a \in A \text{ such that } a <_K b \iff B \subset A + \interior(K), \\
    \text{set-less:} \quad & A \prec_K B \iff A \prec^{u}_K B \text{ and } A \prec^{\ell}_K B.\\
    \end{split}
    \end{equation*}
    In this paper, we will use the set-less relation. However, all the definitions regarding set relations can also be formulated using the upper or lower set-less relations. 
    
    We say that~$A \in \mathcal{F}$ is a \textit{strictly minimal element} of~$\mathcal{F}$ w.r.t.~$K$ if
    \begin{equation}\label{eq:s10}
        \forall B \in \mathcal{F} \text{ such that } B \preceq_K A, \text{ we have } A \preceq_K B.
    \end{equation}
    Note that~\eqref{eq:s10} corresponds to~\eqref{eq:v10} in terms of set relations. 
    We say that~$A \in \mathcal{F}$ is a \textit{(weakly) minimal element} of~$\mathcal{F}$ w.r.t.~$K$ if
    \begin{equation}\label{eq:s20}
        \nexists B \in \mathcal{F} \text{ such that } B \prec_K A.
    \end{equation}
    Note that~\eqref{eq:s20} corresponds to~\eqref{eq:v20} in terms of set relations. 

\subsubsection{Vector-valued optimization and stochastic multi-objective optimization}

    Given a vector function~$F: \mathbb{R}^n \to \mathbb{R}^m$ and a closed, convex, and pointed cone~$K \subset \mathbb{R}^m$, let us consider the following vector-valued optimization~(VVO) problem
    \begin{equation}\label{prob:vvo}
    \le_K \shortdash \min _{x \in X} F(x) \; = \; (f_1(x), \ldots, f_m(x)),
    \end{equation}
    where~$X \subseteq \mathbb{R}^n$. When~$K=\mathbb{R}^m_{+}$, a~VVO problem can also be referred to as a multi-objective optimization problem.
    Denoting the image of the set~$X$ under the vector function~$F$ as~$F(X) = \cup_{x \in X} \{F(x)\}$.
    We say that~$x^* \in X$ is a \textit{strictly efficient solution} (or strict Pareto optimal solution, if~$K=\mathbb{R}^m_{+}$) to problem~\eqref{prob:vvo} if~$F(x^*)$ is a strictly minimal element of~$F(X)$ w.r.t.~$K$.
    We say that~$x^* \in X$ is a \textit{weakly efficient solution} (or weak Pareto optimal solution, if~$K=\mathbb{R}^m_{+}$) to problem~\eqref{prob:vvo} if~$F(x^*)$ is a weakly minimal element of~$F(X)$ w.r.t.~$K$. When~$K=\Rmbb_+$, the Pareto front of problem (2.6) is defined as the image of the set of Pareto optimal solutions under the vector function~$F$. Each point on the Pareto front is referred to as a nondominated point.
 
     Assume that~$K=\mathbb{R}^m_{+}$ and each objective function~$f_i(\cdot)$ is defined as the expected value of a stochastic function~$h_i(\cdot, \xi)$, which depends on a random variable~$\xi$ within a probability space with probability measure independent of~$x$. One can formulate problem~\eqref{prob:vvo} as the following stochastic~MOO problem
     
    \vspace{-0.3cm}
    \begin{equation}
         \min_{x \in X} \; F(x) = (f_1(x), \ldots, f_m(x)) \; = \;  (\mathbb{E}[h_1(x,\xi)], \ldots, \mathbb{E}[h_m(x, \xi)]), \vspace{-0.05cm}
    \label{stochastic_moo}
    \end{equation}
    Problem~\eqref{stochastic_moo} can be solved by applying stochastic approximation~(SA) methods.
    The earliest prototypical~SA algorithm is known as the \textit{stochastic gradient~(SG) algorithm}~\cite{HRobbins_SMonro_1951,KLChung_1954,JSacks_1958}, designed for single-objective problems (i.e., $m=1$). Its update step at the~$k$-th iteration is defined by~$x_{k+1} = x_k - \alpha_k g_i(x_k, \xi_k)$, where~$g_i(x_k,\xi_k)$ is a stochastic gradient generated according to~$\xi_k$ (e.g., ~$g_i(x_k,\xi_k) = \nabla h_i(x_k, \xi_k)$) and $\alpha_k$ is a positive stepsize.    

\begin{figure*}[ht]
  \centering
  \vskip 0cm
  \subfloat[Partial order relations in $\mathbb{R}^2$ \\ for $K = \mathbb{R}^2_+$ (note that $A \prec_K B$).]{\begin{tikzpicture}

\draw[->] (0,0) -- (4.1,0) node[right] {};
\draw[->] (0,0) -- (0,4) node[above] {};

\draw[fill=purple!20] (2,1.5) ellipse [x radius=0.5, y radius=1.0];
\node at (2,1.5) {$\text{A}$};

\draw[fill=green!20] (3.5,2.5) ellipse [x radius=0.4, y radius=0.8];
\node at (3.5,2.5) {$\text{B}$};

\draw[fill=purple!20] (1.1,1.5) ellipse [x radius=0.3, y radius=0.4];
\node at (2,1.5) {$\text{A}$};
\draw[fill=purple!20] (0.7,2.0) ellipse [x radius=0.5, y radius=0.2];
\draw[fill=blue!20] (0.35,2.8) ellipse [x radius=0.2, y radius=0.4];
\draw[fill=blue!20] (3.2,0.8) ellipse [x radius=0.4, y radius=0.3];

\draw[dashed, purple, thin] (1.5,1.5) -- (1.5,4);

\draw[dashed, purple, thin] (2.0,0.5) -- (4.1,0.5);

\draw[dashed, darkgreen, thin] (3.9,2.5) -- (3.9,0);

\draw[dashed, darkgreen, thin] (3.4,3.3) -- (0,3.3);




\node[right] at (0.02, 0.9) {\tiny \textcolor{purple}{Strictly}};
\node[right] at (0.02, 0.7) {\tiny \textcolor{purple}{Minimal Elements}};
\node[right] at (0.5, 3.0) {\tiny \textcolor{blue}{(Weakly)}};
\node[right] at (0.5, 2.8) {\tiny \textcolor{blue}{Minimal Element}};

\end{tikzpicture}} \hspace{0.3cm}
  \subfloat[When~$m=1$, $x_*$ is a strictly minimal solution of $S$ w.r.t. $K = \mathbb{R}_+$.]{\begin{tikzpicture}

\draw[->] (0,0) -- (5.4,0) node[above] {$x$};
\draw[->] (0,0) -- (0,4) node[right] {$S$};

\draw[thick] plot[domain=1:5] (\x,{(\x-3)*(\x-3)/4 + 2.5}) node[right] {};
\draw[thick] plot[domain=1:5] (\x,{(\x-3)*(\x-3)/4 + 0.5}) node[right] {};

\fill[pattern=north east lines, pattern color=black, opacity=0.4] (1,{(1-3)*(1-3)/4 + 0.5}) -- plot[domain=1:5] (\x,{(\x-3)*(\x-3)/4 + 0.5}) -- (5,{(5-3)*(5-3)/4 + 2.5}) -- plot[domain=5:1] (\x,{(\x-3)*(\x-3)/4 + 2.5}) -- cycle;

\draw[solid, thick,  purple] (3,0.5) -- (3,2.5) node[above, black] {\color{purple} $S(x_*)$};
\draw[solid, thick,  darkgreen] (4.5,1.0625) -- (4.5,3.0625) node[above, black] {\color{darkgreen} $S(\bar{x})$};

\filldraw[black] (3,0) circle (2pt) node[below] {$x_*$};
\filldraw[black] (4.5,0) circle (2pt) node[below] {$\bar{x}$};
\node at (1, 2) {$Graph(S)$};





\end{tikzpicture}} \hspace{0.3cm}
    \subfloat[Illustration of the first inequality in~\eqref{eq:set_less_vector}.]{\begin{tikzpicture}

\draw[->] (0,0) -- (4.1,0) node[right] {};
\draw[->] (0,0) -- (0,4) node[above] {};

\draw[fill=purple!20] (2,1.5) ellipse [x radius=0.5, y radius=1.0];
\node at (2,1.5) {$\text{A}$};

\draw[fill=green!20] (3.5,2.5) ellipse [x radius=0.4, y radius=0.8];
\node at (3.5,2.5) {$\text{B}$};





\draw[orange, thick] (2.85, -0.525*2.85 + 1.5) -- (0, 1.5) node[right, font=\tiny, color=black] {$\inf_{a \in A} \ell^\top a$}; 
\draw[orange, thick] (4.1, -0.525*4.1 + 3.5) -- (0, 3.5) node[right, font=\tiny, color=black] {$\inf_{b \in B} \ell^\top b$};

\draw[orange, thick, ->] (0, 0) -- (0.4, 0.8) node[right] {$\ell$}; 

\draw[red, thick] (0, 0) -- (0, 1); 
\draw[red, thick] (0, 0) -- (1, 0); 

\draw [red, thick] (0.3, 0) coordinate (angle) arc[start angle=0, end angle=90, radius=0.3cm];
\draw [red, thick] (0.6, 0) coordinate (angle) arc[start angle=0, end angle=90, radius=0.6cm];

\node[text width= 2cm, font=\scriptsize] at (1.65, 0.2) { \textcolor{red}{$K^* = \mathbb{R}^2_+$}};

\draw (-0.05, 1.5) -- (0.05, 1.5); 
\draw (-0.05, 3.5) -- (0.05, 3.5); 

\end{tikzpicture}}\vspace{0.1cm}
  \caption{Key concepts in~SVO.}\label{fig:SVO}
\end{figure*}
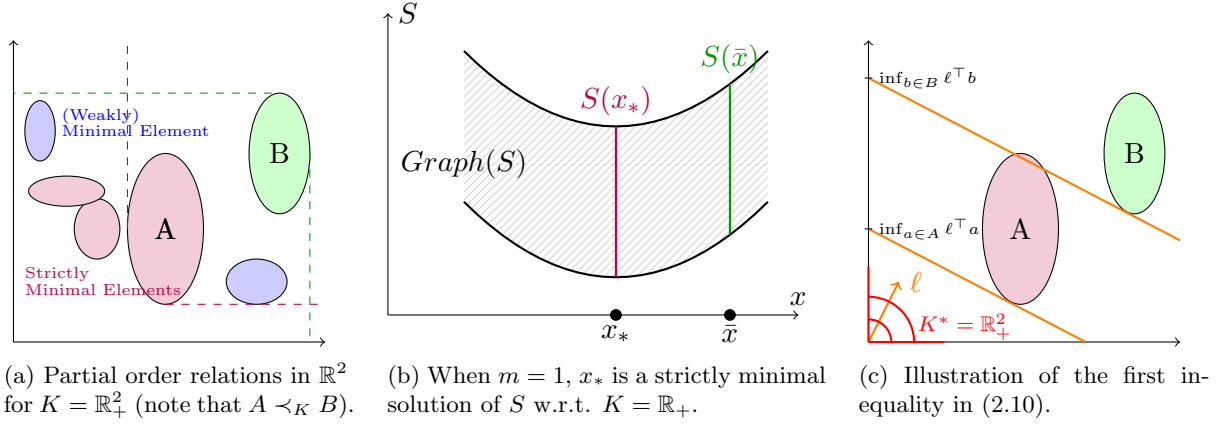

\subsubsection{Set-valued mappings and set-valued optimization}

    Given~$X \subseteq \mathbb{R}^n$, a set-valued mapping~$S: X \rightrightarrows \mathbb{R}^m$ associates a set~$S(x) \subset \mathbb{R}^m$ with each~$x \in X$. The graph of~$S$ is denoted as~$Graph(S) = \{(x,u) \; | \; u \in S(x)\} \subseteq X \times \mathbb{R}^m$.  
    When~$S(x)$ is a singleton for all~$x \in X$, $S$ is referred to as a single-valued function.
    
    Let us consider the~SVO problem in~\eqref{prob:svo}.  
    Moreover, let us denote the image of the set~$X$ under the set-valued function~$S$ as~$S(X) = \cup_{x \in X} S(x)$. In the literature, two approaches have been proposed to define the optimal solution of such a problem: the set approach and the vector approach. In the set approach, we say that~$x^* \in X$ is a \textit{strictly (weakly) minimal solution} to problem~\eqref{prob:svo} if~$S(x^*)$ is a strictly (weakly) minimal element of the family of subsets~$\{S(x)\}_{x \in X}$ w.r.t.~$K$ (see Fig.~\ref{fig:SVO}~(b) for an example).
    
    In the vector approach, we say that~$x^* \in X$ is a strictly (weakly) minimal solution of problem~\eqref{prob:svo} if~$S(x^*)$ contains a strictly (weakly) minimal element of~$S(X)$ w.r.t.~$K$ in the vector sense. In this paper, we focus on the set approach, which is the most popular~\cite{GBouza_2021}. 

\subsection{Vectorization for set-valued optimization} 

Vectorization reduces an~SVO problem to a~VVO problem by reducing the comparison of two sets to the pointwise comparison of suitable vector functions~\cite{JJahn_2015}. Vectorization for~SVO plays the same role that scalarization plays for~VVO. Scalarization reduces a~VVO problem to a real-valued optimization problem by weighting the components of the vector function into a single objective or minimizing one objective subject to the others up to certain thresholds~\cite{MEhrgott_2005}.

Let~$K \subseteq \mathbb{R}^m$ be a closed, convex, and pointed cone. Let~$K^*$ be the dual cone of~$K$, i.e., $K^* = \{\ell \in \mathbb{R}^m ~|~ \ell^\top y \ge 0, \ \forall y \in K\}$. Note that the dual cone of~$K = \mathbb{R}^m_{+}$ is~$K^* = \mathbb{R}^m_{+}$. Moreover, let~$\mathcal{F}$ be a family of subsets of $\mathbb{R}^m$ and let us consider non-empty sets~$A \in \mathcal{F}$ and~$B \in \mathcal{F}$. Assuming that both~$A+K$ and~$B-K$ are closed and convex, we have from~\cite[Lemma 2.1 and Theorem~2.1]{JJahn_2015} (see Fig.~\ref{fig:SVO}~(c)) that
\begin{alignat}{2}
A \preceq^{\ell}_K B \quad & \Longleftrightarrow \quad \inf _{a \in A} \ell^\top a \leq \inf _{b \in B} \ell^\top b, \ \forall \ell \in K^* \backslash\left\{0\right\}, \label{eq:ltype_vector}\\
A \preceq^{u}_K B \quad & \Longleftrightarrow \quad \sup _{a \in A} \ell^\top a \leq \sup _{b \in B} \ell^\top b, \ \forall \ell \in K^* \backslash\left\{0\right\}, \label{eq:utype_vector}\\
A \preceq_K B \quad & \Longleftrightarrow \quad \inf _{a \in A} \ell^\top a \leq \inf _{b \in B} \ell^\top b \ \text{ and } \ \sup _{a \in A} \ell^\top a \leq \sup _{b \in B} \ell^\top b, \ \forall \ell \in K^* \backslash\left\{0\right\}. \label{eq:set_less_vector}
\end{alignat}

Based on~\cite[Corollary~2.2]{JJahn_2015}, one can apply definition~\eqref{eq:s10} with~\eqref{eq:set_less_vector} to obtain a characterization for sets that are strictly minimal elements. In particular, we have that a set~$A$ is a strictly minimal element of~$\mathcal{F}$ w.r.t.~$K$ if and only if there is no~$B$ such that
\begin{equation}\label{theorem:corollary_2.2-1}
    \begin{alignedat}{2}
        &\sup_{b \in B} \ell^{\top} b \le \sup_{a \in A} \ell^{\top} a \ \text{ and }~\inf_{b \in B} \ell^{\top} b \le \inf_{a \in A} \ell^{\top} a, \text{ for all }~\ell \in K^* \backslash\left\{0\right\}, \text{ and }\\
        &\sup_{b \in B} \ell^{\top} b < \sup_{a \in A} \ell^{\top} a \ \text{ or }~\inf_{b \in B} \ell^{\top} b < \inf_{a \in A} \ell^{\top} a, \ \text{ for some } \ell \in K^* \backslash\left\{0\right\}.
    \end{alignedat}
\end{equation}

Instead of~$K^* \backslash\left\{0\right\}$, one can consider~$\tilde{K} = K^* \cap \{y \in \mathbb{R}^m ~|~ \|y\| = 1\}$, and then a discretization~$\{\ell_1, \ldots, \ell_p\} \subset \tilde{K}$. Therefore, as a consequence of~\eqref{eq:set_less_vector}, one has to solve~$4$ optimization subproblems (possibly in parallel) for each~$\ell_i$, with~$i \in \{1,\ldots,p\}$. When~$A$ and~$B$ are polyhedrons (which includes the case of intervals), the subproblems are~LPs and the discretization~$\{\ell_1, \ldots, \ell_p\}$ can be coarser than in the general case~\cite{JJahn_DFO_2015}.

For any non-empty set~$A \in \mathcal{F}$, let us define the following vector-valued functions 
\begin{equation}\label{eq:v_func}
\begin{alignedat}{2}
v_{\ell}(A) ~&=~ \left(\begin{array}{c}
\sup _{a \in A} \ell^\top a \\
\inf _{a \in A} \ell^\top a
\end{array}\right), \forall \ell \in K^* \backslash\left\{0\right\}, \ \text{ and } \\
v(A) ~&=~ (v_{\ell}(A) ~|~ \ell \in K^* \backslash\left\{0\right\}).
\end{alignedat}
\end{equation}
Note that~$v(\cdot)$ is a vector-valued function with an infinite number of component functions. Let us now introduce the following cone
\[
K_{\Pi} = \prod_{\ell \in K^* \backslash\left\{0\right\}} K_{\ell}, \text{ with } K_{\ell} = \mathbb{R}_{+}^2,
\]
which is given by the Cartesian product of an infinite number of two-dimensional non-negative orthants. Based on~\eqref{eq:ltype_vector}--\eqref{eq:set_less_vector} and~\eqref{eq:v_func}, vectorization reduces the comparison of two sets to the pointwise comparison of vector functions as follows 
\begin{equation}\label{eq:1000}
A \preceq_K B \quad
\Longleftrightarrow \quad v_{\ell}(A) \le_{\mathbb{R}_+^2} v_{\ell}(B), \ \forall \ell \in K^* \backslash\left\{0\right\} \quad
\Longleftrightarrow \quad v(A) \le_{K_{\Pi}} v(B),
\end{equation}
where~$A \in \mathcal{F}$ and~$B \in \mathcal{F}$ are non-empty sets.

We can now include Theorem~\ref{theorem:vectorization} below, provided in~\cite[Theorem 3.1]{JJahn_2015}, which states that solving an~SVO problem amounts to solving a corresponding~VVO problem.  

\begin{theorem}[Vectorization theorem]\label{theorem:vectorization}
Let~$S(x)+K$ and~$S(x)-K$ be closed and convex for all~$x \in X$. Then, $\bar{x} \in X$ is a strictly minimal solution to problem~\eqref{prob:svo} if and only if~$\bar{x}$ is a strictly efficient solution to the 
following~VVO problem 
\begin{equation}\label{prob:vectorization}
\le_{K_{\Pi}} \shortdash \min _{x \in X} V(x), \quad \text{ where }~V(x) = v(S(x)), \text{ for all }~x \in X.
\end{equation}
\end{theorem}


In problem~\eqref{prob:vectorization}, the vector-valued objective function is given by an infinite number of component functions. In~\cite{JJahn_DFO_2015}, the author shows that the number of component functions is finite if the sets~$S(x)$ are polyhedral, which includes the case when the sets~$S(x)$ are given by intervals or rectangles.

\section{Hyperbox-valued optimization} \label{sec: hyperbox}
In this section, we consider the general hyperbox-valued optimization (HVO) problem, where the dimension is $m = k \in \mathbb{N}^+$ and $k \geq 2$. Given mappings
\[
F_{lb}:\mathbb{R}^n \rightarrow \mathbb{R}^k, 
\quad F_{ub}:\mathbb{R}^n \rightarrow \mathbb{R}^k,
\]
satisfying $F_{lb}(x) \leq_K F_{ub}(x)$ for all $x \in X$ (where $K = \mathbb{R}^k_+$), we define a hyperbox as follows
\begin{equation*}
[F_{lb}(x),F_{ub}(x)] 
= (\{F_{lb}(x)\} + \mathbb{R}_+^k) 
\cap (\{F_{ub}(x)\} - \mathbb{R}_+^k).
\end{equation*}
The HVO problem is then formulated as
\begin{equation} \label{HVO}
\preceq_K \text{-} \min_{x\in X} S(x) 
\quad\text{with}\quad S(x) = [F_{lb}(x),F_{ub}(x)].
\end{equation}

The vectorization theorem for SVO stated in Theorem~\ref{theorem:vectorization} says that a set-valued optimization problem is equivalent to a vector-valued optimization problem with a number of objectives that can be infinite. In particular cases like the HVO, this number can be shown to be finite. For that purpose, we first describe a result (first stated in~\cite{hernandez2017some} and later proved in~\cite{GEichfelder_TGerlach_2019}) explaining how to directly compare any two hyperboxes of the type defined above.

\begin{proposition}[Comparison of hyperboxes] \label{prop:hyper_compare}
Let $A, B \subseteq \mathbb{R}^k$ be hyperboxes,
\[
A \coloneqq [a^1,b^1] = (\{a^1\} + \mathbb{R}_+^k) \cap (\{b^1\} - \mathbb{R}_+^k), \quad
B \coloneqq [a^2,b^2] = (\{a^2\} + \mathbb{R}_+^k) \cap (\{b^2\} - \mathbb{R}_+^k),
\]
with $a^1,a^2,b^1,b^2 \in \mathbb{R}^k$ satisfying $a^1 \leq_K b^1$ and $a^2 \leq_K b^2$.  
Then:
\[
A \preceq_K^l B \iff a^1 \leq_K a^2
\quad\text{and}\quad
A \preceq_K^u B \iff b^1 \leq_K b^2.
\]
\end{proposition} 

The following result describes the equivalent vector-valued optimization problem for HVO. As noted in~\cite{GEichfelder_TGerlach_2019}, its proof follows directly from Proposition~\ref{prop:hyper_compare}.

\begin{theorem}[Vectorization for HVO]\label{thm:vectorize_HVO}
Consider the HVO problem~\eqref{HVO} with hyperboxes $S(x)=[F_{lb}(x),F_{ub}(x)] = (\{F_{lb}(x)\}+\Rmbb_+^k)\cap(\{F_{ub}(x)\}-\Rmbb_+^k)$.
Let $K=\Rmbb_+^k$.
Then the following statements are equivalent:
\begin{enumerate}
\item[(i)] $\bar{x}\in X$ is a strictly minimal solution to~\eqref{HVO} with respect to $K$.
\item[(ii)] $\bar{x}\in X$ is a strictly efficient solution of the following vector-valued optimization problem
\begin{equation}\label{eq:vec-KK}
\le_{K\times K}\text{-}\min_{x\in X}
\begin{pmatrix}
F_{ub}(x)\\[2pt]
F_{lb}(x)
\end{pmatrix},
\end{equation}
with respect to~$K\times K$.
\end{enumerate}
In particular, since we are considering~$K\times K=\Rmbb_+^{2k}$, the vector-valued optimization problem~\eqref{eq:vec-KK} is further equivalent to
\begin{equation}\label{eq:vec-R2k}
\le_{\Rmbb_+^{2k}}\text{-}\min_{x\in X}
\begin{pmatrix}
F_{ub}(x)\\[2pt]
F_{lb}(x)
\end{pmatrix}.
\end{equation}
\end{theorem}

We note that the general HVO framework naturally includes two known specific set-valued optimization problems. 
When~$k=1$, the hyperboxes reduce to real intervals, and problem~\eqref{HVO} becomes the standard interval-valued optimization (IVO) problem. Given~$f_{\ell b}: \mathbb{R}^n \to \mathbb{R}$, $f_{u b}: \mathbb{R}^n \to \mathbb{R}$, and~$f_{\ell b}(x) \le f_{u b}(x)$, for all~$x \in X$, an~IVO problem can be formulated as follows
\begin{equation}\label{prob:ivo}
    \preceq_{\Rmbb_+} \text{-} \min_{x\in X} S(x) 
\quad\text{with}\quad S(x) = \{y \in \mathbb{R} ~|~ f_{\ell b}(x) \le y \le f_{u b}(x)\}.
\end{equation}
For IVO problems, Theorem~\ref{thm:vectorize_HVO} simplifies to the following corollary.

\begin{corollary}\label{corollary:vectorization}
Let~$K = \mathbb{R}_{+}$. Then, $\bar{x} \in X$ is a strictly minimal solution to problem~\eqref{prob:ivo} w.r.t.~$K$ if and only if~$\bar{x}$ is a strictly efficient solution of the following vector-valued optimization problem
\begin{equation}\label{prob:vvo_interval_3}
\le_{\mathbb{R}_{+}^2} \shortdash \min_{x \in X} ~ (f_{\ell b}(x), f_{u b}(x)).
\end{equation}
\end{corollary}

When $k=2$, the hyperboxes are rectangles in~$\Rmbb^2$, so the HVO problem~\eqref{HVO} reduces to the rectangular-valued optimization (RVO) problem. Given~$F_{\ell b}: \mathbb{R}^n \to \mathbb{R}^2$, $F_{u b}: \mathbb{R}^n \to \mathbb{R}^2$, and~$F_{\ell b}(x) \le_K F_{u b}(x)$, for all~$x \in X$, an~RVO problem can be formulated as follows 
    \begin{equation}\label{prob:sqvo}
    \preceq_{\Rmbb^2_+} \text{-} \min_{x\in X} S(x) 
\quad\text{with}\quad S(x) = [F_{\ell b}(x), F_{ub}(x)].
    \end{equation} 
Similarly, for~RVO problems, the vectorization theorem~\ref{thm:vectorize_HVO} for the HVO simplifies to the following corollary.
\bcorollary \label{corollary:sqvo}
Let $K=\Rmbb_+^2$. Then $\Bar{x}\in X$ is a strictly minimal solution to problem~\eqref{prob:sqvo} w.r.t~$K$ if and only if $\Bar{x}$ is a strictly efficient solution of the following vector-valued optimization problem
\begin{equation} \label{R4}
    \leq_{\Rmbb^4_+} \text{-} \min_{x\in X} \bpmatrix F_{ub}(x)\\
    F_{\ell b}(x) \epmatrix.
\end{equation}
\ecorollary

\section{Stochastic hyperbox-valued optimization}\label{sec:stoch_bvo}

We will introduce new stochastic hyperbox-valued optimization formulations based on super and sub-quantiles. Such formulations will be used for~ML robustness in Section~\ref{sec:robust_learning_bvo}.    

\subsection{Superquantiles and subquantiles} \label{subsec:review}

Let us consider a probability space~$(\Omega, \mathcal{F}, \mathbb{P})$, with probability measure~$\mathbb{P}$ and~$\sigma$-algebra~$\mathcal{F}$ containing subsets of the sample space~$\Omega$. 
Let~$Z$ be a random variable defined in~$\Omega$.
Denoting~$F_Z: \mathbb{R} \to [0,1]$ as the cumulative distribution function of~$Z$ (i.e., $F_Z(z) = \mathbb{P}(Z \le z)$ for all~$z \in \mathcal{Z}$) and~$Q_Z: (0,1) \to \mathbb{R}$ as the quantile function of~$Z$ (i.e., $Q_Z(p) = \min\{z ~|~ F_Z(z) \ge p\}$ for all~$p \in (0,1)$), we can define the superquantile\footnote{In risk management and finance, the quantile and superquantile functions are also known as Value-at-Risk~(VaR) and Conditional Value-at-Risk~(CVaR), respectively~\cite{RTRockafellar_OJohannes_2013}.} of~$Z$ as a function~$\bar{\mathbb{S}}_p$ given by 
\begin{equation}\label{def:superquantile-gen}
\bar{\mathbb{S}}_p[Z] \; = \; \frac{1}{1-p} \int_{p}^{1} Q_Z(p')dp', \text{ for all } p \in (0,1) 
\end{equation}
and the subquantile of~$Z$ as a function~$\underline{\mathbb{S}}_p[Z]$ given by
\begin{equation}\label{def:subquantile-gen}
\underline{\mathbb{S}}_p[Z] \; = \; \frac{1}{p} \int_{0}^{p} Q_Z(p')dp', \text{ for all } p \in (0,1). 
\end{equation}

One can address the cases where~$p=0$ or~$p=1$ by extending the definition in~\eqref{def:superquantile-gen} as follows~(see~\cite{RTRockafellar_OJohannes_2013})
\begin{equation}\label{def:superquantile-ext}
\bar{\mathbb{S}}_0[Z] = \mathbb{E}[Z] \ \text{ and } \ \bar{\mathbb{S}}_1[Z] = \sup_{\omega \in \Omega} Z(\omega).
\end{equation}
Similar to~\eqref{def:superquantile-ext}, one can extend the definition in~\eqref{def:subquantile-gen} for~$p=0$ and~$p=1$ by setting
\begin{equation}\label{def:subquantile-ext}
\underline{\mathbb{S}}_0[Z] = \inf_{\omega \in \Omega} Z(\omega) \ \text{ and } \ \underline{\mathbb{S}}_1[Z] = \mathbb{E}[Z].
\end{equation}

When~$Z$ has a continuous probability distribution, we can give characterizations equivalent to~\eqref{def:superquantile-gen} for the superquantile and~\eqref{def:subquantile-gen} for the subquantile. In particular, the superquantile of~$Z$ reduces to
\begin{equation}\label{def:superquantile}
\bar{\mathbb{S}}_p[Z] \; = \; \mathbb{E}[Z ~|~ Z \ge Q_Z(p)], \text{ for all } p \in (0,1), 
\end{equation}
which shows that~$\bar{\mathbb{S}}_p[Z]$ is the mean of the upper~$p$-tail distribution of~$Z$.
Analogously, the subquantile of~$Z$ reduces to
\begin{equation}\label{def:subquantile}
\underline{\mathbb{S}}_p[Z] \; = \; \mathbb{E}[Z ~|~ Z \le Q_Z(p)], \text{ for all } p \in (0,1), 
\end{equation}
which shows that~$\underline{\mathbb{S}}_p[Z]$ is the mean of the lower $p$-tail distribution of~$Z$.
 
An equivalent characterization of the superquantile function~\eqref{def:superquantile-gen} that applies when~$Z$ has a general probability distribution is (see~\cite[Theorem 1]{RTRockafellar_SUryasev_2000})
\begin{equation}\label{def:superquantile2}
\bar{\mathbb{S}}_p[Z] \; = \; \min_{\eta \in \mathbb{R}} \big\{ \eta + \frac{1}{1-p} \mathbb{E}[\max(Z - \eta, 0)] \big\}.
\end{equation}
The objective function of the optimization problem in~\eqref{def:superquantile2} is non-differentiable and convex in~$\eta$~\cite[Section 3.2]{SSarykalin_etal_2008}. The optimal solution in~$\eta$ is equal to~$Q_Z(p)$. Moreover,$~\bar{\mathbb{S}}_p[Z]$ is convex in~$Z$.

One can derive a similar characterization for the subquantile function~$\underline{\mathbb{S}}_p[Z]$ (see~Appendix~\ref{app:sub}) as follows:
\begin{equation}\label{def:uryasev}
\underline{\mathbb{S}}_p[Z] \; = \; \max_{\eta \in \mathbb{R}} \big\{ \eta - \frac{1}{p} \mathbb{E}[\max(\eta - Z, 0)] \big\}.
\end{equation}
Similarly, the optimal solution in~$\eta$ is equal to~$Q_Z(p)$. However,$~\underline{\mathbb{S}}_p[Z]$ is now a concave function in~$\eta$ and~$Z$.

\subsection{Stochastic interval-valued optimization}\label{subsec:stoch_ivo}

Let us consider a continuously differentiable function~$\psi: \mathbb{R}^n \times \mathbb{R}^d \to \mathbb{R}$ and two scalars~$\underline{p} \in [0,1]$ and~$\bar{p} \in [0,1]$ (we include the boundaries in such intervals to account for the extended definitions provided in~\eqref{def:superquantile-ext} and~\eqref{def:subquantile-ext}). In this paper, we focus on stochastic~IVO problems that can be written as follows
\begin{equation}\label{prob:stochastic_ivo}
\text{ problem~\eqref{prob:ivo} with } f_{\ell b}(x) = \underline{\mathbb{S}}_{\underline{p}}[\psi(x,\xi)] \text{ and } f_{u b}(x) = \bar{\mathbb{S}}_{\bar{p}}[\psi(x,\xi)],
\end{equation}
where~$\xi \in \mathbb{R}^d$ is a random variable having a probability distribution independent of~$x$.

As a consequence of Corollary~\ref{corollary:vectorization} in Section~\ref{sec: hyperbox}, the set of minimal solutions to the~IVO problem~\eqref{prob:ivo} can be determined by solving the corresponding multi-objective problem~\eqref{prob:vvo_interval_3}, where $f_{\ell b}(x)$ and $f_{u b}(x)$ are given by~\eqref{prob:stochastic_ivo}.
One can consider different cases based on the values of~$\underline{p}$ and~$\bar{p}$. In Sections~\ref{subsubsec:case_not1},~\ref{subsubsec:case_not01}, and~\ref{subsubsec:case1}, we focus on the most significant ones, from which all the remaining cases can be trivially obtained. 

\subsubsection{Case~$\underline{p} = 1$ and~$\bar{p} = 0$}\label{subsubsec:case_not1}
Recalling the extended definitions~\eqref{def:superquantile-ext} and~\eqref{def:subquantile-ext}, the case~$(\underline{p} = 1, \bar{p} = 0)$ results in a bi-objective optimization problem where both objective functions are given by equal expectations, i.e., $f_{\ell b}(x) = f_{ub}(x) = \mathbb{E}[\psi(x,\xi)]$. Such a problem can solved by applying the classical stochastic gradient method~\cite{LBottou_FECurtis_JNocedal_2018}.

\subsubsection{Case~$\underline{p} = 0$ and~$\bar{p} = 1$}\label{subsubsec:case_not01}
When~$\psi(x,\xi)$ is a continuous function of~$x$ and~$\xi$, the case~$(\underline{p} = 0, \bar{p} = 1)$ leads to~IVO problems for which the equivalent bi-objective problem~\eqref{prob:vvo_interval_3} can be written as
\begin{equation}\label{prob:determinisic_ivo}
\le_{\mathbb{R}_{+}^2} \shortdash \min_{x \in X} ~ (f_{\ell b}(x) = \min_{\xi \in P}\psi(x,\xi),~ f_{u b}(x) = \max_{\xi \in P}  \psi(x,\xi)),
\end{equation}
where~$P \subseteq \mathbb{R}^d$ represents the set of values taken on by~$\xi$, assumed compact. Note that problem~\eqref{prob:determinisic_ivo} is deterministic. Danskin's Theorem~\cite{JMDanskin_1967} provides a way to compute the subdifferentials of~$f_{ub}$ and~$f_{\ell b}$. In particular, if~$\psi(x,\xi)$ is differentiable in~$x$ for all~$\xi \in P$ and~$\nabla_x \psi(x,\xi)$ is a continuous function of~$\xi$ for all~$x \in X$, then the subdifferential of~$f_{ub}$ is
\[
\partial f_{ub}(x) = \conv\{\nabla_x \psi(x,\xi) \; | \; \xi \in P_{ub}(x)\}, \text{ where } P_{ub}(x) = \{\bar{\xi} \in P \; | \; \psi(x,\bar{\xi}) = \max_{\xi \in P} \; \psi(x,\xi) \}.
\]
The subdifferential of~$f_{\ell b}$ and the corresponding set~$P_{\ell b}$ can be obtained from similar steps. 
To solve problem~\eqref{prob:determinisic_ivo}, one can apply a multi-subgradient method (and by that we mean the multi-gradient method using subgradients instead of gradients~\cite{SLiu_LNVicente_2021}) or a multi-gradient sampling method (and by that we mean the extension of gradient sampling~\cite{JVBurke_ASLewis_MLOverton_2005} to MOO).

\subsubsection{Case~$\underline{p} \in (0,1)$ and~$\bar{p} \in (0,1)$}\label{subsubsec:case1}
We will first focus on~$f_{ub}$. 
Assume that one has access to a set of~$N$ i.i.d. samples of~$\xi$, denoted as~$\{\xi_1, \ldots, \xi_N\}$. 
One can now approximate the superquantile~$\bar{\mathbb{S}}_{\bar{p}}[\psi(x,\xi)]$ in~$f_{ub}(x)$ using data samples, resulting in
\begin{equation}\label{eq:superquantile2_approx}
f_{ub}^N(x) \; = \; \bar{\mathbb{S}}^N_{\bar{p}}[\psi(x,\xi)] = \min_{\eta \in \mathbb{R}} \big\{ \eta + \frac{1}{N(1-\bar{p})} \sum_{i=1}^{N}[\max(\psi(x,\xi_i) - \eta, 0)] \big\},
\end{equation}
which is obtained from the empirical version of~\eqref{def:superquantile2} by setting~$Z = \psi(x,\xi)$ and~$p=\bar{p}$. 
Similarly, for~$f_{\ell b}$, the empirical approximation of the subquantile can be written as
\begin{equation}\label{eq:subquantile_approx}
  f_{\ell b}^N(x)
  \;=\;
  \underline{\mathbb{S}}_{p}^N\!\big[\psi(x,\xi)\big]
  \;=\;
  \max_{\eta\in\Rmbb}\Big\{
    \eta \;-\; \frac{1}{Np}\sum_{i=1}^N [\max(\,\eta-\psi(x,\xi_i),0)]
  \Big\}.
\end{equation}


\subsection{Stochastic rectangle-valued optimization}\label{subsec:stoch_rvo}
In this section, let us consider two continuously differentiable functions~$\psi_1$ and~$\psi_2$ (both from~$\mathbb{R}^n \times \mathbb{R}^d$ to~$\mathbb{R}$) with two scalars~$\underline{p}$ and $\bar{p}$. In this paper, we are focusing on the following stochastic~RVO problem
\begin{equation}\label{prob:stochastic_rvoo}
\text{ problem~\eqref{prob:sqvo} with } F_{\ell b}(x) = (\underline{\mathbb{S}}_{\underline{p}}[\psi_1(x,\xi)], \bar{\mathbb{S}}_{\bar{p}}[\psi_1(x,\xi)]) \text{ and } F_{u b}(x) = (\underline{\mathbb{S}}_{\underline{p}}[\psi_2(x,\xi)], \bar{\mathbb{S}}_{\bar{p}}[\psi_2(x,\xi)]), 
\end{equation}
where~$\xi\in\Rmbb^d$ is a random variable having a probability distribution independent of~$x$.
Following Corollary~\ref{corollary:sqvo} in Section~\ref{sec: hyperbox}, determining the minimal solutions of~\eqref{prob:stochastic_rvoo} reduces to solving the multi-objective problem~\eqref{R4}, with the four objectives specified by~\eqref{prob:stochastic_rvoo}. For the stochastic RVO problem, one can also consider different cases based on the values of~$\underline{p}$ and~$\bar{p}$, which we have discussed in Section~\ref{subsec:stoch_ivo}.

\section{Robust learning through stochastic hyperbox-valued optimization}\label{sec:robust_learning_bvo}

To apply set-valued optimization in machine learning~(ML), one can reformulate an~ML problem as an~SVO problem where a set-valued empirical risk function is minimized instead of a traditional finite-sum function. The set-valued empirical risk function, denoted as~$L(x) \subseteq \mathbb{R}$, associates a set of loss function values (computed over all the points in the training set) with each~$x \in \mathbb{R}^n$, where~$x$ represents the vector of parameters of an~ML model (e.g., weights of a neural network). A solution obtained through set-valued~ML can result in lower errors on testing data compared to traditional~ML because the set-valued empirical risk function considers both the expected value and variance of the loss function across all training data points, while traditional empirical risk focuses solely on the expected value. 

The setting we consider in this paper is binary supervised classification. Let us denote a features/labels dataset used for training by~$\mathcal{D}_{train} = \{ (u_i, v_i), \, i \in \{ 1, \ldots, N\} \}$, which consists of~$N$~pairs of a real feature vector~$u_i \in \mathbb{R}^d$ and a target label~$v_i \in \{-1,+1\}$. We assume that for any data point~$i \in \{1, \ldots, N\}$, the classification is deemed correct if the right label is predicted. To evaluate the loss incurred when using the prediction function~$\phi(u; x)$, where~$x \in \mathbb{R}^n$ is a vector of parameters, we use a loss function~$\ell(\phi(u; x), v)$. 

In traditional~ML, we train the model by minimizing the single-valued empirical risk function~$\mathcal{L}(x) = (1/N) \sum_{i=1}^N \ell(\phi (u_i; x), v_i)$, which computes the average loss over the samples in~$\mathcal{D}_{train}$.
A different perspective consists of minimizing a set-valued empirical risk mapping~$L(x) = \{\ell(\phi (u_i; x), v_i)~|~ i \in \{1, \ldots, N\}\}$, which computes
a set of loss values over all training points. Note that~$\mathcal{L}:\mathbb{R}^n \to \mathbb{R}$ and~$L:\mathbb{R}^n \rightrightarrows \mathbb{R}$. We will develop stochastic~IVO and RVO formulations based on super and sub-quantiles, which provide flexibility in capturing likely scenarios while avoiding extreme ones that may not reflect practical outcomes.

\subsection{Robust learning through stochastic interval-valued optimization}
\par\noindent
\begin{wrapfigure}[]{R}{0.40\textwidth}
    \centering
    \begin{tikzpicture}[scale=0.7]
    \draw[->] (-0.5, 0) -- (5, 0) node[above] {\footnotesize $\xi$};
    \draw[->] (0, -0.5) -- (0, 4); 
    \node at (2.5, 3.5) {\color{blue}\footnotesize $\psi_\xi = \psi(x,\xi)$};
    

    \def\xbar{2.5}
    
    \draw[thick, domain=0.82:\xbar, smooth, variable=\x, blue] 
        plot ({\x}, {2*(\x-2)^2+0.2});
    
    \def\slope{4*(\xbar-2)}
    
    \def\fvalue{2*(\xbar-2)^2+0.2}
    
    \def\a{-1}
    \def\b{\slope}
    \def\c{\fvalue}
    
    \draw[thick, domain=\xbar:3.5, smooth, variable=\x, blue] 
        plot ({\x}, {\a*(\x-\xbar)^2 + \b*(\x-\xbar) + \c});

    \draw[thick, domain=3.5:4.62, smooth, variable=\x, blue] 
        plot ({\x}, {(\x-3.5)^2 + 1.7});

    \draw[dashed, red] (0, 3.0) -- (4.62, 3.0);
    \node at (-0.2, 3.3) [left, red] {\footnotesize $\max_{\xi} \psi_\xi$};

    \draw[dashed, red] (0, 0.2) -- (2, 0.2);
    \node at (-0.2, 0.3) [left, red] {\footnotesize $\min_{\xi} \psi_\xi$};

    \draw[ultra thick, black] (0.82, 0) -- (4.62, 0);
    \node at (2.72, -0.3) {\footnotesize $\Xi$};

    \draw[ultra thick, black] (0, 0.6) -- (0, 2.62);
    
    \node[right] at (0, 0.6) {\color{darkgreen} \footnotesize $\underline{\mathbb{S}}_{\underline{p}}[\psi_\xi]$};
    
    \node[right] at (0, 2.62) {\color{darkgreen} \footnotesize $\bar{\mathbb{S}}_{\bar{p}}[\psi_\xi]$};
    
    \node at (-0.7, 1.6) {\footnotesize $S(x)$};

    \draw[thick, darkgreen] (-0.1, 0.6) -- (0.1, 0.6);
    
    \draw[thick, darkgreen] (-0.1, 2.62) -- (0.1, 2.62);

\end{tikzpicture}
    \caption{$S(x)$ for problem~\eqref{prob:stochastic_ivo}.}
    \label{fig:example_ivo}
\end{wrapfigure}
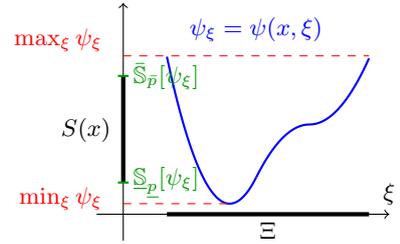

\noindent Let us consider a continuously differentiable function~$\psi: \mathbb{R}^n \times \mathbb{R}^d \to \mathbb{R}$ and two scalars~$\underline{p} \in [0,1]$ and~$\bar{p} \in [0,1]$ (we include the boundaries in such intervals to account for the extended definitions of superquantiles and subquantiles provided in~\eqref{def:superquantile-ext}). 
Let~$\xi$ 
be a random vector taking values in~$\Xi \subseteq \mathbb{R}^d$ with a probability distribution independent of~$x$. We propose reformulating any~ML problem involving a single objective (e.g., minimizing misclassification error) by using the stochastic~IVO problem~\eqref{prob:stochastic_ivo},
where~$\psi(x,\xi)$ plays the role of the loss function~$\ell(\phi(u;x),v)$, with~$\xi = (u,v)$. Figure~\ref{fig:example_ivo} provides an illustration for~$d = 1$.
As part of this section, we will use the bi-objective optimization problem~\eqref{prob:vvo_interval_3} to study the robustness of the optimal solution to this stochastic~IVO problem in popular~ML applications. Remark~\ref{remark} motivates the proposed robust IVO approach by providing an example that interprets the geometric meaning of Corollary~\ref{corollary:vectorization} and contrasts the resulting solution set with both ERM and the classical min--max robust solution.

\newpage
\par\noindent
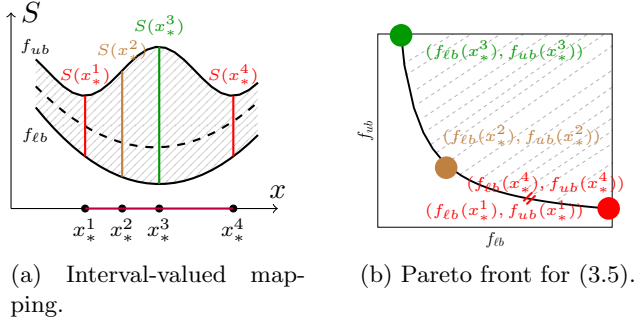
\begin{wrapfigure}{l}{0.60\textwidth}
  \centering
  \vskip 0cm
    \subfloat[Interval-valued mapping.]{\begin{tikzpicture}[scale=0.65]

\draw[->] (0,0) -- (5.4,0) node[above] {$x$};
\draw[->] (0,0) -- (0,4) node[right] {$S$};

\draw[thick] plot[domain=0.5:5] (\x,{2.8 + sin(deg(2.1*\x-4.7))*0.5}) node[right] {};
\draw[thick] plot[domain=0.5:5] (\x,{(\x-3)*(\x-3)/4 + 0.5}) node[right] {};

\draw[thick, dashed] plot[domain=0.5:5] (\x,{(\x-3)*(\x-3)/7 + 1.6 + 0.07*((\x-0.5)*(\x-5))}) node[right] {};

\fill[pattern=north east lines, pattern color=black, opacity=0.4]
  (0.5,{(0.5-3)*(0.5-3)/4 + 0.5}) -- 
  plot[domain=0.5:5] (\x,{(\x-3)*(\x-3)/4 + 0.5}) -- 
  (5,{2.8 + sin(deg(2.1*5-4.7))*0.5}) -- 
  plot[domain=5:0.5] (\x,{2.8 + sin(deg(2.1*\x-4.7))*0.5}) -- cycle;

\draw[solid, thick,  red] (1.5,1.0625) -- (1.5,2.3) node[above, black] {\color{red}\tiny $S(x_*^1)$};
\draw[solid, thick,  brown] (2.25,0.65) -- (2.25,2.8) node[above, black] {\color{brown}\tiny $S(x_*^2)$};
\draw[solid, thick,  darkgreen] (3,0.5) -- (3,3.3) node[above, black] {\color{darkgreen}\tiny $S(x_*^3)$};
\draw[solid, thick,  red] (4.5,1.0625) -- (4.5,2.3) node[above, black] {\color{red}\tiny $S(x_*^4)$};

\filldraw[black] (1.5,0) circle (2pt) node[below] {\scriptsize $x_*^1$};
\filldraw[black] (2.25,0) circle (2pt) node[below] {\scriptsize $x_*^2$};
\filldraw[black] (3,0) circle (2pt) node[below] {\scriptsize $x_*^3$};
\filldraw[black] (4.5,0) circle (2pt) node[below] {\scriptsize $x_*^4$};

\draw[thick, purple] (1.5, 0) -- (4.5, 0);

\node at (0.5, 3.3) {\tiny $f_{ub}$};
\node at (0.5, 1.5) {\tiny $f_{\ell b}$};

\end{tikzpicture}}
    \hspace{0.5cm}
    \subfloat[Pareto front for~\eqref{prob:vvo_interval_3}.]{\begin{tikzpicture}[scale=0.45]
\centering
\begin{axis}
[xmin=0, xmax=3,
ymin=-0.5, ymax=3,
xlabel=$f_{\ell b}$,
ylabel = $f_{ub}$,
ylabel near ticks,
xlabel near ticks,
xlabel style={font=\Large},
ylabel style={font=\Large},
axis on top=true,
xtick=\empty,
ytick=\empty,
domain=0.1:3,]
\addplot [mark=none,draw=black,ultra thick, solid] {1/(x) - 0.5}; 
\addplot[style={solid}, draw=black, dashed, ultra thin, opacity=0.3] coordinates {(2.3, 0) (8, 5.4)};
\addplot[style={solid}, draw=black, dashed, ultra thin, opacity=0.3] coordinates {(2, 0.1) (8, 5.6)};
\addplot[style={solid}, draw=black, dashed, ultra thin, opacity=0.3] coordinates {(1.85, 0.15) (8, 6)};
\addplot[style={solid}, draw=black, dashed, ultra thin, opacity=0.3] coordinates {(1.65, 0.2) (8, 6.2)};
\addplot[style={solid}, draw=black, dashed, ultra thin, opacity=0.3] coordinates {(1.48, 0.25) (8, 6.6)};
\addplot[style={solid}, draw=black, dashed, ultra thin, opacity=0.3] coordinates {(1.3, 0.35) (8, 7)};
\addplot[style={solid}, draw=black, dashed, ultra thin, opacity=0.3] coordinates {(1.05, 0.45) (8, 7.2)};
\addplot[style={solid}, draw=black, dashed, ultra thin, opacity=0.3] coordinates {(0.87, 0.65) (8, 7.8)};
\addplot[style={solid}, draw=black, dashed, ultra thin, opacity=0.3] coordinates {(0.75, 0.85) (8, 8.1)};
\addplot[style={solid}, draw=black, dashed, ultra thin, opacity=0.3] coordinates {(0.68, 1) (8, 8.3)};
\addplot[style={solid}, draw=black, dashed, ultra thin, opacity=0.3] coordinates {(0.65, 1.27) (8, 8.27)};
\addplot[style={solid}, draw=black, dashed, ultra thin, opacity=0.3] coordinates {(0.6, 1.5) (8, 8.5)};
\addplot[style={solid}, draw=black, dashed, ultra thin, opacity=0.3] coordinates {(0.53, 1.75) (8, 8.75)};
\addplot[style={solid}, draw=black, dashed, ultra thin, opacity=0.3] coordinates {(0.5, 2.0) (8, 9)};
\addplot[style={solid}, draw=black, dashed, ultra thin, opacity=0.3] coordinates {(0.35, 2.25) (8, 9)};
\addplot[style={solid}, draw=black, dashed, ultra thin, opacity=0.3] coordinates {(0.32, 2.45) (8, 9)};


\end{axis}

\foreach \Point in {(0.68, 5.6)}{
    \node at \Point {\textcolor{darkgreen}{\scalebox{2}{\textbullet}}};
}

\foreach \Point in {(2,1.7)}{
    \node at \Point {\textcolor{brown}{\scalebox{2}{\textbullet}}};
}

\foreach \Point in {(6.75, 0.50)}{
    \node at \Point {\textcolor{red}{\scalebox{2}{\textbullet}}};
}

\node[text width=1cm] at (2.5, 5.2) {\color{darkgreen} \tiny $(f_{\ell b}(x_*^3),f_{ub}(x_*^3))$};
\node[text width=1cm] at (3, 2.7) {\color{brown} \tiny $(f_{\ell b}(x_*^2),f_{ub}(x_*^2))$};
\node[text width=1cm] at (3.7, 1.3) {\color{red} \tiny $(f_{\ell b}(x_*^4),f_{ub}(x_*^4))$};
\node[rotate=45] at (4.5, 0.8) {\color{red} \texttt{=}};
\node[text width=1cm] at (2.5, 0.5) {\color{red} \tiny $(f_{\ell b}(x_*^1),f_{ub}(x_*^1))$};

\node[text width=1cm] at (4, 4) { };

\end{tikzpicture}}
  \caption{Vectorization theorem for~IVO.}\label{fig:example_ivo_vectorization}
\end{wrapfigure}

\begin{remark}\label{remark}    Figure \ref{fig:example_ivo_vectorization} 
    illustrates Corollary \ref{corollary:vectorization} and problem \eqref{prob:vvo_interval_3} when $f_{\ell b}(x) = \underline{\mathbb{S}}_{\underline{p}}[\psi(x,\xi)]$ and $f_{u b}(x) = \bar{\mathbb{S}}_{\bar{p}}[\psi(x,\xi)]$, where $\psi(x,\xi)$ represents the training loss function. In Figure \ref{fig:example_ivo_vectorization}(a), all points between~$x_*^1$ and~$x_*^4$ are minimal solutions of~$S$ w.r.t.~$\mathbb{R}_+$. Such points correspond to the strict Pareto optimal solutions of problem~\eqref{prob:vvo_interval_3}, as shown in Figure~\ref{fig:example_ivo_vectorization}(b). Note that~$x_*^3$ is the minimizer of the dashed line in Figure~\ref{fig:example_ivo_vectorization}(a), which one can interpret as the single-valued expected training loss function~$\mathbb{E}[\psi(x,\xi)]$. Such a minimizer results in training loss values that are either very low (risk of overfitting) or very high (risk of underfitting). Additionally, $x_*^1$ and~$x_*^4$ are local minimizers of~$f_{ub}$ and represent the solutions from the traditional robust optimization approach (min-max formulation), which focuses on the worst-case scenario. In contrast, points like~$x_*^2$ offer a better compromise and are expected to perform more effectively on validation data.
\end{remark}

\subsection{Robust learning through stochastic rectangle-valued optimization} \label{sec: rl through rvo}

We claim that any~ML application problem involving two objectives (e.g., minimizing both misclassification error and social bias) can be formulated as a stochastic~RVO problem~\eqref{prob:stochastic_rvoo}. 
In addition to the accuracy loss function~$\ell$, one can consider a second objective, such as a fairness loss function~$\ell_F$, resulting in an additional set-valued empirical risk function~$L_F(x) ~=~ \{\ell_F(\phi (u_i; x), v_i) ~|~ i \in \{1, \ldots, N\}\}$. 
Note that here we focus on robustness to biased datasets, but other loss functions can be used to address different forms of robustness or fairness.
Given two continuously differentiable functions~$\psi_1$ and~$\psi_2$ (both from~$\mathbb{R}^n \times \mathbb{R}^d$ to~$\mathbb{R}$),~$\underline{p}$, $\bar{p}$, and~$\xi$, we propose reformulating a fair~ML problem by using 
the four-objective optimization problem~\eqref{R4} with $F_{\ell b}(x) = (\underline{\mathbb{S}}_{\underline{p}}[\psi_1(x,\xi)], \bar{\mathbb{S}}_{\bar{p}}[\psi_1(x,\xi)]) \text{ and } F_{u b}(x) = (\underline{\mathbb{S}}_{\underline{p}}[\psi_2(x,\xi)], \bar{\mathbb{S}}_{\bar{p}}[\psi_2(x,\xi)])$ ,
where~$\psi_1$ and~$\psi_2$ represent the accuracy and fairness loss functions.

\subsection{Numerical results}\label{section: numerical}
\subsubsection{Solving the optimization/learning problems}\label{sec:sol}
First, as a baseline for comparison, we consider the empirical risk minimization (ERM) model obtained by minimizing the empirical risk $\mathcal{L}(x)$. We denote an ERM solution by $x_{\mathrm{ERM}}$, defined as
\begin{equation*}
    x_{\mathrm{ERM}} \;\in\; \arg\min_{x}\, \mathcal{L}(x)
    \;=\;
    \arg\min_{x}\; \frac{1}{N}\sum_{i=1}^{N}\ell\bigl(\phi(u_i;x),\,v_i\bigr).
\end{equation*}
We compute $x_{\mathrm{ERM}}$ using the standard stochastic gradient (SG) method.

Next, we need to select~$x_{\mathrm{IVO}}$ as a certain Pareto optimal solution of the proposed bi-objective problem~\eqref{prob:vvo_interval_3} for~$f_{\ell b}(x)=\underline{\mathbb{S}}_{\underline{p}}[\psi(x,\xi)]$ and~$f_{u b}(x)=\bar{\mathbb{S}}_{\bar{p}}[\psi(x,\xi)]$, where~$\psi(x,\xi)$ is the loss function~$\ell(\phi(u;x),v)$ with~$\xi = (u,v)$. 
However, the set of Pareto optimal solutions typically contains many candidate solutions, and it is therefore neither practical nor principled to choose an arbitrary point in each test instance. Instead, an automatic selection rule is needed to identify a representative solution in a consistent and reproducible manner, and Pareto knee solutions offer an interesting trade-off as they are verbally characterized as Pareto optimal solutions corresponding to nondominated points on the Pareto front where a small improvement in any objective causes a large deterioration on at least one other objective~\cite{KDeb_SGupta_2010,JBranke_etal_2004,giovannelli2025pareto}.
We select~$x_{\mathrm{IVO}}$ as the knee solution using an angle-based criterion in the objective space. Specifically, after ordering the candidate nondominated points on the approximated front, we choose the knee solution as the Pareto optimal solution corresponding to the point with the smallest interior angle formed by its two adjacent neighbors, as this identifies the sharpest local bend of the Pareto front.

We also denote by~$x_{\mathrm{RVO}}$ an approximate selected Pareto knee solution of the proposed four-objective problem \eqref{R4} with $F_{\ell b}(x) = (\underline{\mathbb{S}}_{\underline{p}}[\psi_1(x,\xi)], \bar{\mathbb{S}}_{\bar{p}}[\psi_1(x,\xi)])$ and $F_{u b}(x) = (\underline{\mathbb{S}}_{\underline{p}}[\psi_2(x,\xi)], \bar{\mathbb{S}}_{\bar{p}}[\psi_2$ $(x,\xi)])$,
where~$\psi_1$ and~$\psi_2$ represent the accuracy and fairness loss functions. To compute~$x_{\mathrm{RVO}}$, we first orthogonally project all 4-dimensional points of the computed Pareto front onto the bi-dimensional space of accuracy and into the bi-dimensional space of fairness (resulting into 2 subsets of bi-dimensional approximated Pareto fronts). Then, for each of them, we calculate a knee point using the procedure described for~$x_{\mathrm{IVO}}$, resulting in 2~knees. For each knee, we then calculate a pair of weights (summing to one), by relating them to the corresponding extreme points. Finally, we sum the two knees in a weighted fashion (normalizing the weights to~1) to obtain $x_{\mathrm{RVO}}$.

To compute the Pareto fronts needed for~$x_{\mathrm{IVO}}$/$x_{\mathrm{RVO}}$, we use the Pareto-front stochastic multi-gradient (PF-SMG) method (see Appendix D). 
In doing so, to handle the inner non-smooth hinge term $\max(\cdot,0)$ in the empirical sub/superquantile-based objectives (4.11) and (4.12), we apply the smoothing stochastic gradient (SSG) method (see Appendix C). Specifically, each objective of the bi-objective and four-objective MOO reformulations derived from (3.5) and (3.7) is smoothed, allowing the computation of stochastic gradients.
This smoothing leads to stable gradient estimates in practice.
For the outer minimization over the auxiliary variable $\eta$ in~\eqref{def:superquantile2} and~\eqref{def:uryasev}, in the finite-sum setting with samples $z_1,\dots,z_n$, we approximate the quantile $Q_z(p)$ by sorting $\{z_i\}_{i=1}^n$ and selecting the empirical $p$-quantile.
For comparison, we also report results obtained with the single-solution stochastic multi-gradient (SMG) method (see Appendix~B), which returns a single Pareto point rather than an entire front.

\subsubsection{Robust evaluation}
To compare ERM and stochastic~IVO/RVO in terms of robustness, we evaluate each model on multiple independent test sets. Specifically, we consider $N_{\mathrm{rep}}$ test replications. For each replication $r\in\{1,\dots,N_{\mathrm{rep}}\}$, we draw an independent test set
\[
    \Dcal_{\mathrm{test}}^{(r)}
    \;=\;
    \bigl\{(u_i^{(r)},v_i^{(r)})\bigr\}_{i=1}^{N_{\mathrm{test}}}
\]
from a possibly shifted (or class-imbalanced) test distribution. More specifically, starting from the original, unmodified test set, we construct class-imbalanced test sets in each replication by adjusting the class proportions via sub-sampling from either the positive or the negative class. This yields test sets in which the fraction of positive instances is either smaller or larger than that in $\Dcal_{\mathrm{train}}$, thereby inducing a controlled distributional shift. Such shifts are common in practice: models are often trained on historical data from a particular region or time period, but deployed on populations whose demographic or economic characteristics differ across regions or evolve over time. Under this setting, our approach is expected to deliver more robust performance in the presence of significant distributional changes.

Let us now introduce the metrics used for the robust evaluation.
Let the empirical test risk for the $r$-th test replication be defined as
\begin{equation*}
\mathrm{Loss}^{(r)}(x)
=
\frac{1}{N_{\mathrm{test}}}
\sum_{i=1}^{N_{\mathrm{test}}}
\mathrm{loss}\bigl(\phi(u_{i}^{(r)};x),\, v_{i}^{(r)}\bigr),
\end{equation*}
where $\mathrm{loss}(\cdot)$ denotes either the standard empirical loss
function $\ell(\cdot)$ used in the ERM model, the weighted loss
$\lambda \underline{\mathbb{S}}_{\underline{p}}[\ell(\cdot)] + (1 - \lambda) \overline{\mathbb{S}}_{\bar{p}}[\ell(\cdot)]$ used in the IVO formulation (where the sub- and super-quantiles are represented as $f_{lb}(x)$ and $f_{ub}(x)$, respectively),
or the fairness loss $\ell_F(\cdot)$ used in the RVO formulation.
The weight $\lambda$ is chosen according to the knee point of the Pareto
front obtained from the IVO model.

Since the testing procedure is repeated for $N_{\mathrm{rep}}$ replications,
the robustness of the candidate models can be assessed by applying
standard statistical measures to the operator $\mathrm{Loss}^{(r)}(\cdot)$
across replications, such as the sample mean and sample variance.
We denote the mean and variance of $\mathrm{Loss}(\cdot)$ as
\[
M(x) =
\frac{1}{N_{\mathrm{rep}}}
\sum_{i=1}^{N_{\mathrm{rep}}}
\mathrm{Loss}^{(i)}(x),
\qquad
V(x) =
\frac{1}{N_{\mathrm{rep}}-1}
\sum_{i=1}^{N_{\mathrm{rep}}}
\bigl(\mathrm{Loss}^{(i)}(x)-M(x)\bigr)^2 .
\]



\subsubsection{Adult dataset results}
We test our methodology on the Adult dataset, a standard benchmark for binary classification. Each sample $u \in \mathbb{R}^d$ represents
demographic and socio-economic attributes of an individual (e.g., age,
education level, occupation, hours per week), and the binary label
$v \in \{0,1\}$ indicates whether the individual’s annual income exceeds
\$50{,}000. After standard preprocessing (one-hot encoding of categorical
variables, normalization of continuous features, and removal of examples
with missing values), we randomly split the data into a training set
$\Dcal_{\mathrm{train}}$ and an initial test set~$\Dcal_{\mathrm{test}}$. And we choose logistic regression as our prediction model. For this test and all other tests in the following sections, we use the same experimental settings. In particular, to ensure a fair comparison, all training procedures are performed with the same computational budget of 1500 iterations, the same initial stepsize of 1.3 (with the same discount rate). For the test phase, we use $N_{\mathrm{rep}}=30$ replications of the test procedure.

\begin{figure}[!htbp]
    \centering
    \subfloat[PF-SMG approximation of the non-dominated set and the selected knee point.]{
        \includegraphics[width=0.4\textwidth]{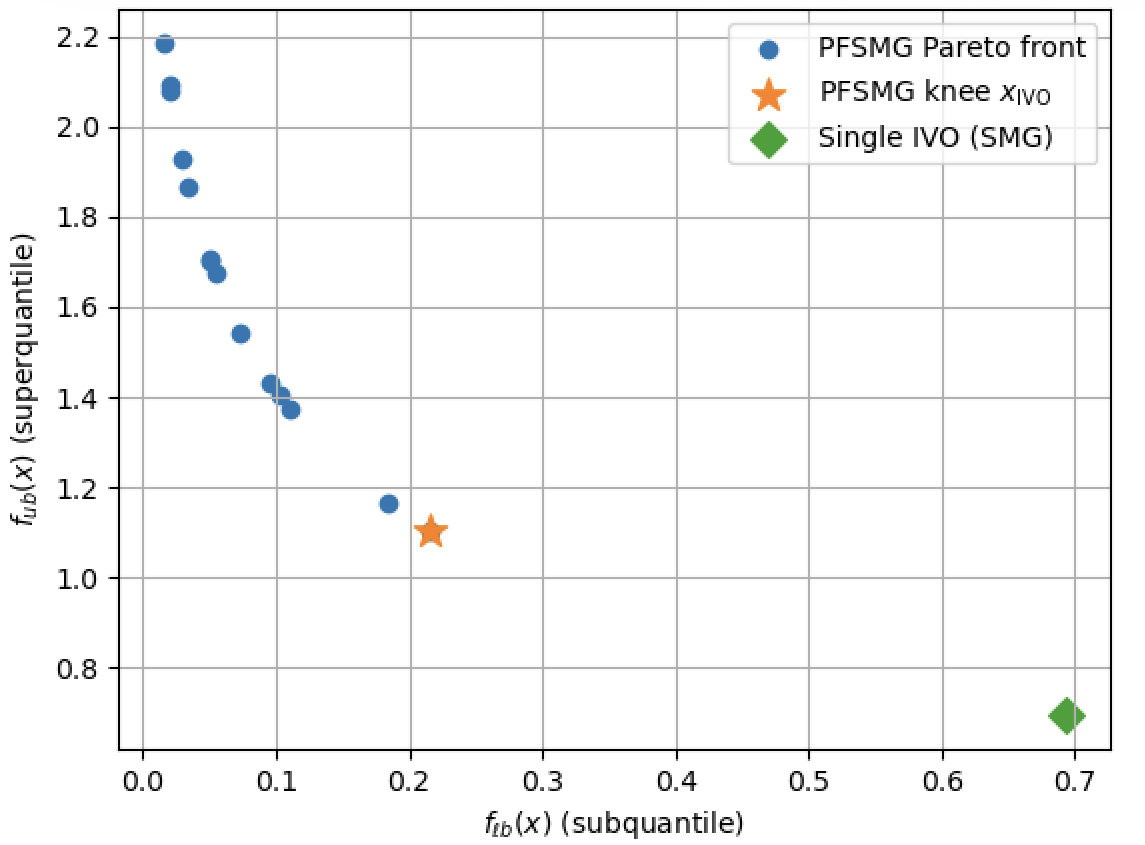}
        \label{fig:adult_pareto}
    }
    \hfill
    \subfloat[Replication variability under different positive-class fractions.]{
        \includegraphics[width=0.48\textwidth]{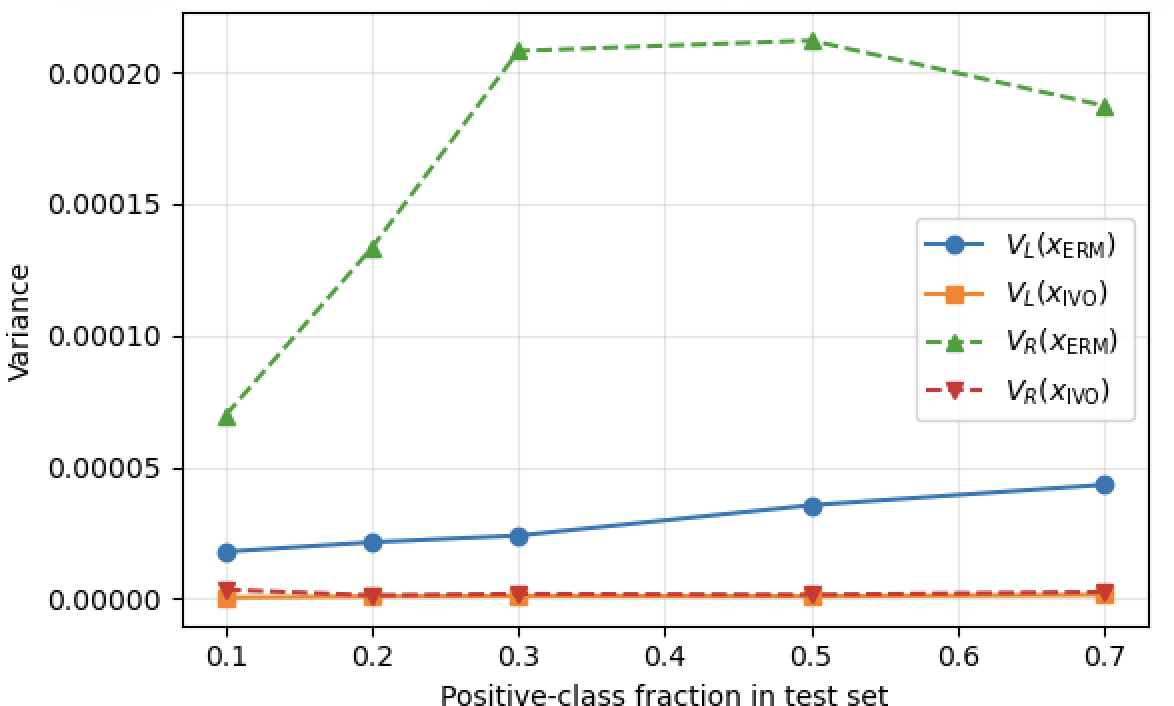}
        \label{fig:adult_variance}
    }
    \caption{Adult dataset results for IVO.}
    \label{fig:adult_results}
\end{figure}

Figure~\ref{fig:adult_pareto} reports the PF-SMG output in the
$\bigl(f_{\ell b}(x),f_{u b}(x)\bigr)$ plane.
The non-dominated set forms a Pareto front, and the chosen knee point
$x_{\mathrm{IVO}}$ (orange star) provides a balanced compromise between the lower- and upper-tail criteria.
For reference, the single-objective IVO solution produced by SMG (green diamond) achieves a much smaller
$f_{u b}(x)$ but with a markedly larger $f_{\ell b}(x)$, illustrating that extreme behavior may arise from a fixed scalarization.
Let us denote by $V_L(x)=V(x)$ the variance evaluated at $x=x_{\mathrm{ERM}}$ with $\mathrm{loss}(\cdot)=\ell(\cdot)$,
and by $V_R(x)=V(x)$ the variance evaluated at $x=x_{\mathrm{IVO}}$
with $\mathrm{loss}(\cdot)=\lambda \underline{\mathbb{S}}_{\underline{p}}[\ell(\cdot)] + (1 - \lambda) \overline{\mathbb{S}}_{\bar{p}}[\ell(\cdot)]$.
Figure~\ref{fig:adult_variance} summarizes stability under altered label proportions.
Across all tested positive-class fractions, $x_{\mathrm{IVO}}$ yields uniformly smaller replication-to-replication variability
than $x_{\mathrm{ERM}}$. 

\subsubsection{German credit dataset results}\label{subsubsec:german-credit}
We next test our methodology on the German credit dataset, also a popular benchmark for binary credit-risk classification. For this dataset,
each sample $u\in\Rmbb^d$ represents applicant-related attributes (e.g., account status, credit history, employment and housing),
and the binary label $v\in\{0,1\}$ indicates good versus bad credit. We apply the same preprocessing pipeline as in the Adult experiments
(one-hot encoding for categorical variables, normalization of continuous features, and removal of examples with missing values),
randomly split the data into a training set $\Dcal_{\mathrm{train}}$ and an initial test set $\Dcal_{\mathrm{test}}$,
and choose logistic regression as the prediction model.

\begin{figure}[!htbp]
    \centering
    \subfloat[PF-SMG approximation of the non-dominated set and the selected knee point.]{
        \includegraphics[width=0.4\textwidth]{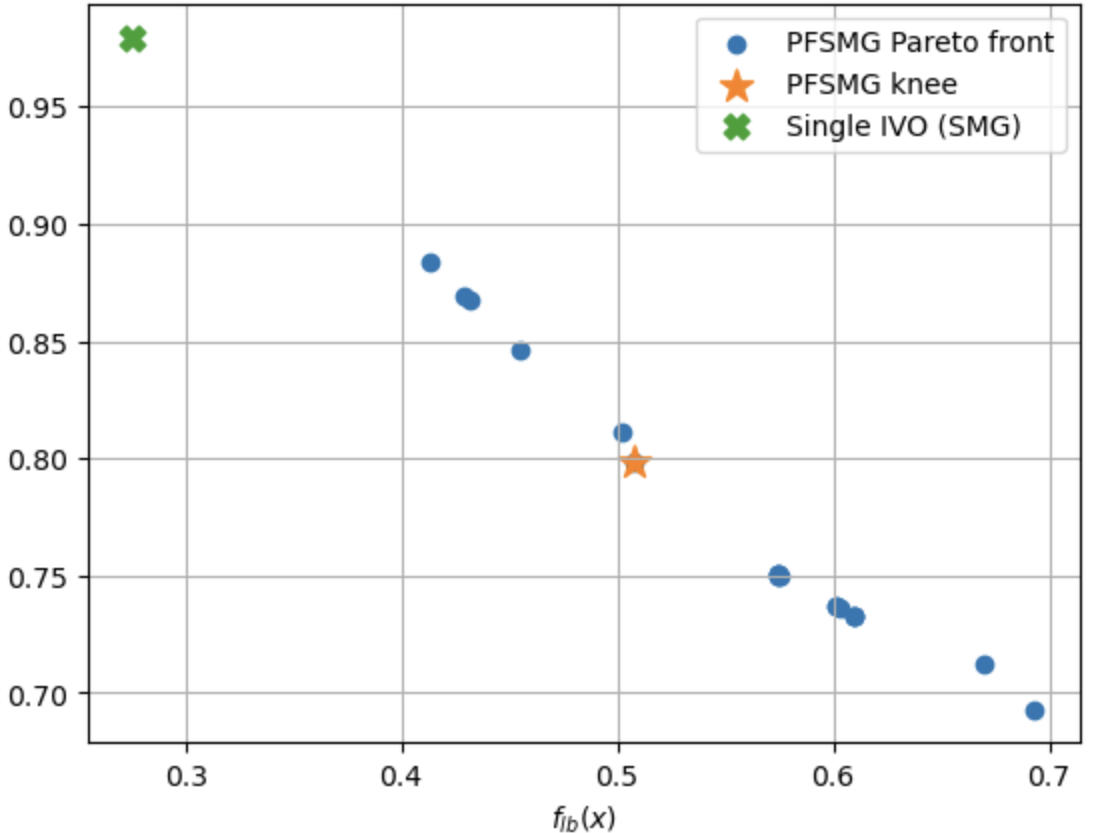}
        \label{fig:german_pareto}
    }
    \hfill
    \subfloat[Stability under different positive-class fractions.]{
        \includegraphics[width=0.48\textwidth]{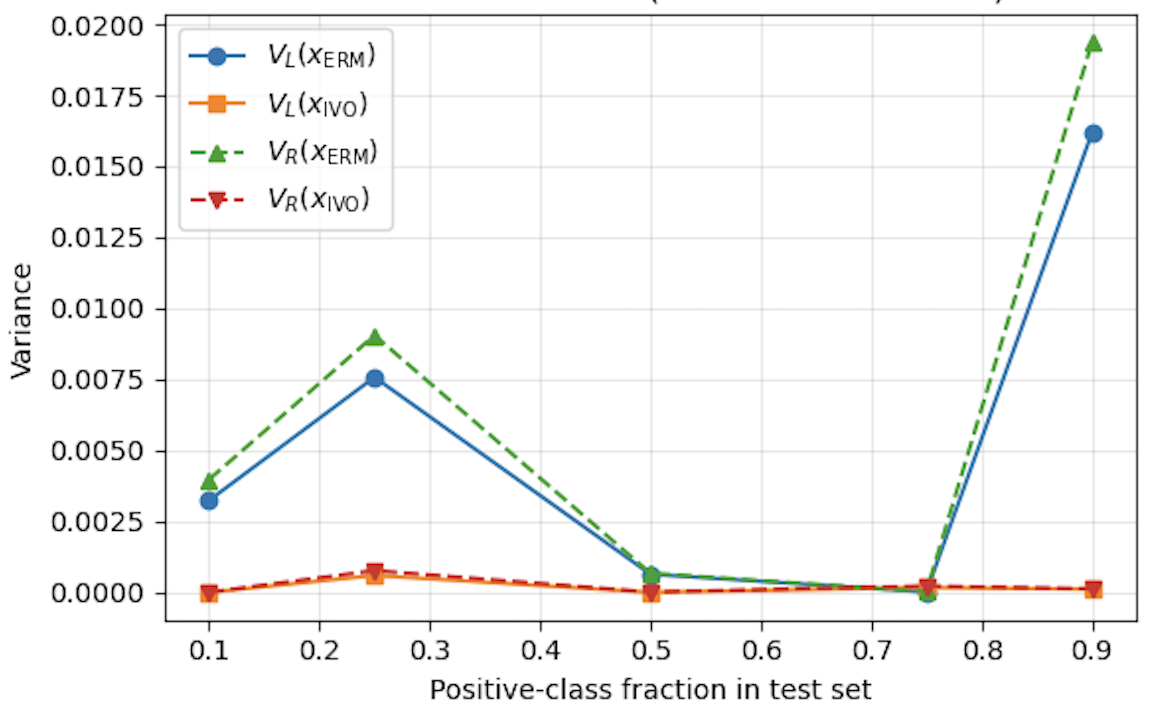}
        \label{fig:german_variance}
    }
    \caption{German Credit dataset results for IVO.}
    \label{fig:german_results}
\end{figure}

Similar to the previous test results, Figure~\ref{fig:german_pareto} visualizes the PF-SMG approximation of the Pareto front in the
$\bigl(f_{\ell b}(x),f_{u b}(x)\bigr)$ space. Figure~\ref{fig:german_variance} summarizes stability under class-imbalance shift.
Using the same evaluation protocol and robustness metrics as in the Adult experiments, we compare $x_{\mathrm{IVO}}$ and $x_{\mathrm{ERM}}$
across a range of positive-class fractions.
Again, $V_L(x)=V(x)$ is the variance
evaluated at $x=x_{\mathrm{ERM}}$ with $\mathrm{loss}(\cdot)=\ell(\cdot)$,
and $V_R(x)=V(x)$ is the variance evaluated at $x=x_{\mathrm{IVO}}$
with $\mathrm{loss}(\cdot)=\lambda \underline{\mathbb{S}}_{\underline{p}}[\ell(\cdot)] + (1 - \lambda) \overline{\mathbb{S}}_{\bar{p}}[\ell(\cdot)]$.
The test results indicate that $x_{\mathrm{IVO}}$ exhibits smaller variability than $x_{\mathrm{ERM}}$ on the German credit dataset as well.

\subsubsection{RVO results}
In this session, we test the stochastic RVO model and compare it with the baseline ERM model. As discussed in Section~\ref{sec: rl through rvo}, beyond the standard accuracy loss~$\ell$, we also incorporate a fairness loss function~$\ell_F$. The motivation is that optimizing accuracy alone can yield classifiers that achieve 
strong overall performance while still exhibiting systematic disparities across sensitive groups. Thus, introducing a fairness 
objective makes this trade-off explicit and allows us to search for solutions that balance predictive performance and equity. 
Moreover, from a robust learning perspective, fairness constraints/objectives may mitigate sensitivity to distributional shift 
(e.g., changes in group proportions or group-conditional feature distributions) by discouraging overly group-dependent decision rules, 
thereby improving model stability and robustness. In particular, we target disparate impact, approximated by a covariance-based 
metric that penalizes statistical dependence between the sensitive attribute and the prediction score~\cite{SLiu_LNVicente_2020}. Concretely, for model parameters $x$, with prediction function $\phi(u_i;x)$, we define the fairness loss
$\ell_F$ (a covariance proxy for disparate impact) by
\[
\ell_F(x)\;:=\; 
\left(\frac{1}{N}\sum_{i=1}^{N}\bigl(a_i-\bar a\bigr)\,\phi(u_i;x)\right)^2,
\qquad
\bar a := \frac{1}{N}\sum_{i=1}^{N} a_i,
\]
where $a_i$ denotes the sensitive attribute associated with $u_i$.

We denote by~$V_F(x)$ the variance evaluated at
both $x_{\mathrm{RVO}}$ and $x_{\mathrm{ERM}}$ with
$\mathrm{loss}(\cdot)=\ell_F(\cdot)$. To summarize robustness with a single scalar
that emphasizes fairness stability while accounting for predictive performance,
we introduce the fairness-risk-per-accuracy (FRPA) metric
\[
\mathrm{FRPA}(x)
=
\frac{V_F(x)}
{M_{\mathrm{ACC}}(x)},
\]
where~$M_{\mathrm{ACC}}(x)
=
\frac{1}{N_{\mathrm{rep}}}
\sum_{i=1}^{N_{\mathrm{rep}}}
\mathrm{Accuracy}^{(i)}(x)$,
and for the $r$-th test replication the accuracy is defined as
\[
\mathrm{Accuracy}^{(r)}(x)
=
\frac{1}{N_{\mathrm{test}}}
\sum_{i=1}^{N_{\mathrm{test}}}
\mathbf{1}\!\left(
\phi(u_i^{(r)};x)=v_i^{(r)}
\right).
\]
The numerator $V_F(x)$ measures how sensitive the fairness loss is to test
resampling under distributional shift, while the denominator
$M_{\mathrm{ACC}}(x)$ measures the average predictive performance across
replications. Hence, smaller values of $\mathrm{FRPA}(x)$ indicate that the
model exhibits more stable fairness behavior across replications while
maintaining competitive predictive accuracy.
\begin{figure}[ht]
    \centering
    \includegraphics[width=0.5\linewidth]{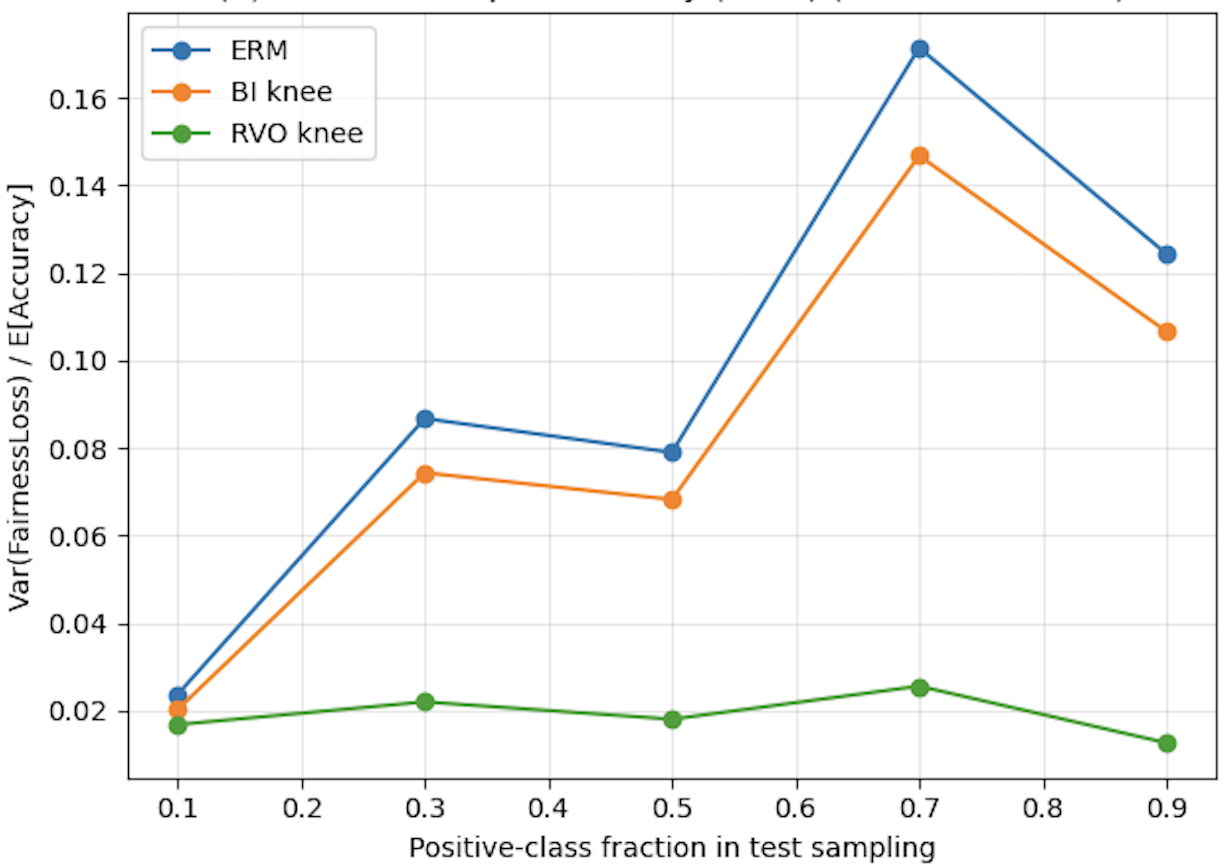}
    \caption{Adult dataset results for RVO.}
    \label{fig:RVO}
\end{figure}

We conducted an experiment on the Adult dataset, taking gender as the sensitive attribute. For comparison, we also solve a standard bi-objective problem involving fairness and predictive performance, without using the subquantile/superquantile criteria, and select its corresponding Pareto knee solution (denoted by BI knee).
Figure~\ref{fig:RVO} reports the performance of ERM, BI knee, and the RVO knee solution under class-proportion shifts using the $\mathrm{FRPA}$ metric. The results show that BI knee uniformly outperforms ERM, while the RVO knee solution attains the lowest $\mathrm{FRPA}$ values across all tested shift levels. Thus, although the bi-objective knee already improves fairness relative to ERM by accounting for both fairness and accuracy rather than accuracy alone, the RVO knee provides the strongest robustness, producing more stable fairness behavior while preserving competitive predictive performance under distributional shift.

\section{Conclusions and future works}\label{sec:conclusion}
This work illustrates that set-valued modeling can be used as a practical tool for robust learning. Specifically, in the hyperbox setting, the set-less relation reduces the comparison of image sets to a multi-objective optimization problem with a finite number of objectives. Consequently, an effective modeling framework depends on constructing meaningful and informative lower and upper objectives. 

We therefore study and explain how subquantiles and superquantiles can serve as suitable lower and upper objectives for this robust learning framework because of their natural ability to capture lower- and upper-tail behavior, thereby providing informative measures of optimistic and pessimistic performance under uncertainty.

We also contributed on the theoretical side by establishing and proving a Rockafellar--Uryasev-type characterization for subquantiles. This result complements the classical optimization representation of superquantiles, integrates naturally into our stochastic robust learning framework, and fills a theoretical gap in the analysis of lower-tail risk measures.

For the numerical results, in the distributional-shift experiments generated by changing the class proportions in the test set, the selected IVO knee solutions typically exhibit smaller sample variances across test replications than ERM, indicating greater robustness under distributional shift. In the accuracy--fairness setting, the robust RVO model provides decision makers with a way to select solutions that account for the robustness of fairness while maintaining competitive predictive accuracy, rather than improving one objective at the expense of the other.

Several extensions follow naturally from the present framework. First, beyond the IVO and RVO constructions studied here, one could incorporate additional performance criteria, such as precision, recall, or other fairness measures, together with accuracy and fairness, thereby leading to higher-dimensional hyperbox-valued models. This would make it possible to address more complex robust learning scenarios and to improve several metrics that are important to stakeholders simultaneously.
Second, although we adopt an angle-based knee selection rule on the Pareto front in this work, alternative selection criteria could be developed to target stability and robustness more directly~\cite{KDeb_SGupta_2010,JBranke_etal_2004,giovannelli2025pareto}. 
Third, the tail levels in the subquantile and superquantile objectives act as hyperparameters. Therefore, developing adaptive or data-driven procedures for selecting and tuning these parameters could further improve the reliability of the resulting decisions.

\section*{Acknowledgments}
This work is partially supported by the U.S. Air Force Office of Scientific Research~(AFOSR) award~FA9550-23-1-0217 and the U.S. Office of Naval Research~(ONR) award~N000142412656.

\bibliography{bibs/ref-set-valued,bibs/ref-BSG,bibs/ref-smg-moo,bibs/ref-bsg-moo,bibs/ref-n,bibs/ref-fairness,bibs/ref}

\appendix
\section{Proof of the subquantile characterization}\label{app:sub}
We now give a proof of the characterization~\eqref{def:uryasev} for subquantile functions, following the notation in Section~\ref{subsec:review}.
We begin with introducing a lemma that states the symmetry relationship between the quantiles of a random variable and those of its negation: at any level~$u\in(0,1)$ where the quantile function~$Q_Z$ is continuous, the $(1-u)$-quantile of~$-Z$ is exactly the negative of the $u$-quantile of~$Z$, that is, $Q_{-Z}(1-u)=-\,Q_Z(u)$.

\begin{lemma} \label{lemma:sym}
Let $Z$ be a real-valued random variable with cumulative distribution function~$F_Z(t)=\mathbb{P}(Z\le t)$ and quantile function~$Q_Z(u)=\inf\{t\in\mathbb{R}:F_Z(t)\ge u\}$ where~$u\in(0,1).$
For every~$u\in(0,1)$ that is a continuity point of the quantile function~$Q_Z$, the following identity holds
\begin{equation} \label{eq:quantile_identity}
 Q_{-Z}(1-u)=-\,Q_Z(u).   
\end{equation}
\end{lemma}
\begin{proof}
First, we write the CDF of $-Z$ 
(since~$F_Z$ is nondecreasing and bounded in~$[0,1]$, the left limit~$F_Z(t^-)=\lim_{s\uparrow t}F_Z(s)$ exists for every $t$)
\begin{equation} \label{eq:F_-z}
F_{-Z}(t)=\mathbb{P}(-Z\le t)=\mathbb{P}(Z\ge -t)=1-\mathbb{P}(Z<-t)=1-F_Z((-t)^-).
\end{equation}
For $u\in(0,1)$, by definition of the quantile function, we have
\begin{equation}\label{eq:Q_-z}
Q_{-Z}(1-u)
=\inf\{t:F_{-Z}(t)\ge 1-u\}
=\inf\{t:F_Z((-t)^-)\le u\}.
\end{equation}

Let $x=Q_Z(u)=\inf\{s:F_Z(s)\ge u\}$. Combing~\eqref{eq:F_-z} and~\eqref{eq:Q_-z}, and using the fact that~$F_Z(x)\ge u$ and $F_Z(x^-)\leq u$ yields
\begin{equation} \label{eq: Fz}
 F_{-Z}(-x)=1-F_Z(x^-)\geq 1-u,   
\end{equation}
which implies $Q_{-Z}(1-u)\le -x=-Q_Z(u)$ (since from~\eqref{eq: Fz}, we know that~$-x$ belongs to the set $\{t: F_{-Z}(t)\ge 1-u\}$).

Assume now that $u$ is a continuity point of the monotone function~$Q_Z$.
We claim that for every $y>x$,
\begin{equation}\label{eq:claim}
F_Z(y^-)>u.
\end{equation}
Indeed, if $F_Z(y^-)\le u$ for some $y>x$, then for every $v\in(u,1)$ we still have
$F_Z(y^-)\le v$. Hence $F_Z(t)<v$ for all $t<y$, which yields $Q_Z(v)\ge y$.
Letting $v\downarrow u$ gives $\lim_{v\downarrow u}Q_Z(v)\ge y>x=Q_Z(u)$, contradicting
continuity at $u$. Thus \eqref{eq:claim} holds.

Now take any $t<-x$ and set $y=-t>x$. By \eqref{eq:claim} we have
\[
F_{-Z}(t)=1-F_Z(y^-)<1-u,
\]
so no $t<-x$ can satisfy $F_{-Z}(t)\ge 1-u$. Then by using~\eqref{eq:Q_-z} again, one has
$Q_{-Z}(1-u)\ge -x=-Q_Z(u)$.

Combining both~$Q_{-Z}(1-u)\le -x=-Q_Z(u)$ and~$Q_{-Z}(1-u)\ge -x=-Q_Z(u)$ yields~$Q_{-Z}(1-u)=-Q_Z(u)$ at every continuity point~$u$ of~$Q_Z$.
\end{proof}

This identity~\eqref{eq:quantile_identity} allows us to rewrite integrals involving~$Q_Z$ in terms of~$Q_{-Z}$. In particular, it provides a direct way to express the subquantile~$\underline{\mathbb{S}}_{p}[Z]$ in terms of the superquantile~$\bar{\mathbb{S}}_{1-p}[-Z]$, which is the key step in proving the characterization below.

\begin{theorem}
    Let~$Z$ be a real-valued random variable. For~$p\in(0,1)$, the subquantile of~$Z$ at level~$p$ admits the characterization in~\eqref{def:uryasev}, repeated here for convenience:
\begin{equation}\label{eq:subquantile-uryasev}
\underline{\mathbb{S}}_{p}[Z]
=
\max_{\eta\in\mathbb{R}}
\left\{
\eta - \frac{1}{p}\,\mathbb{E}\big[(\eta-Z)_+\big]
\right\}.
\end{equation}
\end{theorem}

\begin{proof}
By Lemma~\ref{lemma:sym}, the quantile functions of~$Z$ and~$-Z$ satisfy~$Q_{-Z}(1-u)=-Q_Z(u)$.
It is easy to see that~$Q_Z(\cdot)$ is nondecreasing in~$u$.
Therefore, the set of discontinuity points of~$Q_Z$ is at most countable (hence has Lebesgue measure zero), so the identity may fail only on a null set of~$u$, which does not affect the value of the integral.
Thus, recalling the quantile representation of subquantiles~\eqref{def:subquantile-gen} and using the relation above gives
\begin{equation} \label{eq:s^s_}
\underline{\mathbb{S}}_{p}[Z]
= -\frac{1}{p}\int_{0}^{p} Q_{-Z}(1-u)\,du
= -\frac{1}{p}\int_{1-p}^{1} Q_{-Z}(v)\,dv
= -\bar{\mathbb{S}}_{1-p}[-Z],
\end{equation}
where we use the change of variables $v=1-u$ and the definition of $\bar{\mathbb{S}}_{1-p}$ in~\eqref{def:superquantile-gen}.

Now apply Rockafellar--Uryasev's characterization of superquantiles~\eqref{def:superquantile2} to the random variable $-Z$ at level $1-p$, we have
\begin{equation} \label{eq:s^1-p}
\bar{\mathbb{S}}_{1-p}[-Z]
=
\min_{\alpha\in\mathbb{R}}
\left\{
\alpha + \frac{1}{p}\,\mathbb{E}\big[(-Z-\alpha)_+\big]
\right\}.
\end{equation}
Negating both sides of~\eqref{eq:s^1-p} and using the relation of~\eqref{eq:s^s_} yields
\begin{equation*}
\underline{\mathbb{S}}_{p}[Z]
=
\max_{\alpha\in\mathbb{R}}
\left\{
-\alpha - \frac{1}{p}\,\mathbb{E}\big[(-Z-\alpha)_+\big]
\right\}.
\end{equation*}
Finally, by substituting $\eta:=-\alpha$, we obtain
\begin{equation*}
\underline{\mathbb{S}}_{p}[Z]
=
\max_{\eta\in\mathbb{R}}
\left\{
\eta - \frac{1}{p}\,\mathbb{E}\big[(\eta-Z)_+\big]
\right\},
\end{equation*}
which is exactly~\eqref{eq:subquantile-uryasev}.

\end{proof}

\section{Stochastic multi-gradient (SMG) method}\label{SMG}

In this section of the appendix, we briefly describe the stochastic multi-gradient (SMG) method, which extends the traditional single-objective stochastic gradient (SG) method to the multi-objective setting. Let us first consider a deterministic~MOO problem
\begin{equation*}
     \min_{x \in X} \; F(x) = (f_1(x), \ldots, f_m(x)).
\label{deterministic_moo_app}
\end{equation*}
In the smooth (continuously differentiable) case, the steepest descent direction or negative multi-gradient is given by the solution of the subproblem (see~\cite{JFliege_BFSvaiter_2000})
\begin{equation*}
g(x) \in \argmin_{u \in \mathbb{R}^n} \max_{i \in \{1, \ldots, m\}} \nabla f_i(x)^\top u + (1/2) \Vert u\Vert^2,
\end{equation*}
or equivalently, by taking the dual of the subproblem, seen 
as the minimal-norm convex combination of the~$\nabla f_i(x)$'s,
\begin{equation*}
    \begin{aligned}
    g(x) = \sum_{i=1}^m \lambda^i \nabla f_i(x) \quad \mbox{with $\lambda$ the solution of} \quad 
  \operatorname{min}_{\lambda} \; & \big\Vert
 \sum_{i=1}^m \lambda^i \nabla f_i(x)
  \big\Vert^{2} \quad \text{ s.t. } \quad \lambda \in \Delta,
    \end{aligned}
\end{equation*}
where
$\Delta = \{ \lambda \in \mathbb{R}^m: \sum_{i=1}^m \lambda^i = 1, \lambda^i \geq 0, i=1,\ldots,m\}$ is the simplex set. Given a stepsize~$\alpha_k>0$, the corresponding multi-gradient descent update at each iteration~$k$ is
\begin{equation}\label{prob:dMG}
 x_{k+1}=x_k-\alpha_k\,g(x_k).   
\end{equation}

For the stochastic setting, a stochastic~MOO problem can be written as
\begin{equation}
     \min_{x \in X} \; F(x) = (f_1(x), \ldots, f_m(x)) \; = \;  (\mathbb{E}[h_1(x,\xi)], \ldots, \mathbb{E}[h_m(x, \xi)]),
\label{stochastic_moo_app}
\end{equation}
where each objective function~$f_i(\cdot)$ is defined as the expected value of a stochastic function~$h_i(\cdot, \xi)$, which depends on a random variable~$\xi$ within a probability space with probability measure independent of~$x$. Solving the stochastic~MOO problem~\eqref{stochastic_moo_app} at each iteration~$k$, the stochastic multi-gradient~(SMG) algorithm~\cite{MQuentin_PFabrice_JADesideri_2018,SLiu_LNVicente_2021} takes a step along a stochastic multi-gradient~$g(x_k, \xi_k)$
(obtained by replacing the gradients in~\eqref{prob:dMG} by their stochastic estimates)
\begin{equation*}\label{prob:SMG}
    \begin{aligned}
    g(x_k, \xi_k) = \sum_{i=1}^m \lambda_k^i g_i(x_k, \xi_k) \quad \mbox{with $\lambda_k$ the solution of} \quad 
  \operatorname{min}_{\lambda} \; & \big\Vert
 \sum_{i=1}^m \lambda^i g_i(x_k, \xi_k)
  \big\Vert^{2} \quad \text{ s.t. } \quad \lambda \in \Delta.
    \end{aligned} 
\end{equation*}

\section{Smoothing stochastic gradient (SSG) method}\label{SSG}

Notice that the super- and sub-quantile objectives considered in this paper can be expressed in the compositional form
    \[
    \max~\{ 0, \mathbb{E}[\psi(x,\xi)]) \}.
    \]
In particular, the outer mapping $t \mapsto \max\{0,t\}$ is non-smooth, which makes the resulting objective non-differentiable at the kink point even when the inner function $\psi(\cdot,\xi)$ is smooth.
To handle this outer non-smoothness, we will use smoothing techniques that approximate $\phi(t)=\max\{0,t\}$ by a family of continuously differentiable smoothing functions~$\tilde{\phi}(t,\mu)$ which converge to the original function as the smoothing parameter~$\mu$ goes to zero (at the cost of increasingly larger Lipschitz constants). This approximation motivated the smoothing stochastic gradient (SSG) method introduced in~\cite{GiovannelliTanVicente25T010}, which performs stochastic gradient steps on the smoothed objectives while driving the smoothing parameter to zero at a designated rate.

Next, we briefly present the idea of the SSG algorithm to solve the expected risk problem written in the general form
\begin{equation} \label{problem: f(x)}
    \min_{x\in\Rmbb^n}~f(x)~~(f~\text{involves}~\xi,~\psi~\text{and}~\phi),
\end{equation}
where~$x$ is the vector of decision variable,~$\xi$ is a random vector,~$\psi$ is a differentiable function, and~$\phi$ is a function which is not necessarily smooth. Problem~\eqref{problem: f(x)} is intended to cover, as a special case, objectives such as $\max\{0,\mathbb{E}[\psi(x,\xi)]\}$ discussed above.
We consider that a smoothing function~$\ftilde(x,\mu)$ is available for~$f(x)$ involving a smoothing function~$\tilde{\phi}$ of~$\phi$ and also including the randomness given by~$\xi$. We also assume that one can draw gradient estimates~$\nabla_x\ftilde(x,\xi,\mu)$ for the smoothing function~$\ftilde(x,\mu)$. An SSG method can then be stated in Algorithm~$\ref{algorithm: 1}$ and it only differs from the traditional stochastic gradient algorithm in the update of the smoothing parameter~$\mu_k$. 
\begin{algorithm}[H] 
	\caption{Smoothing stochastic gradient (\textbf{SSG}) method}\label{algorithm: 1}
	\begin{algorithmic}[1]
		\medskip
		\item[] {\bf Input:} $x_1$,~$\{\alpha_k\}$, and~$\{\mu_k\}$.
		\medskip
		\item[] {\bf For~$k = 1, 2, \ldots$ \bf do}
		\item[] \quad {\bf Step 1.} 
        Generate a realization~$\xi_k$ of the random variable~$\xi$. Then compute~$\nabla\ftilde_x(x_k,\xi_k,\mu_k)$.
            \nonumber
            \item[] \quad {\bf Step 2.}
        Update iterate~$x_{k+1} = x_k-\alpha_k\nabla\ftilde_x(x_k,\xi_k,\mu_k)$.
            \nonumber
		\item[] \quad {\bf Step 3.} 
        Update smoothing parameter~$\mu_{k+1}\leq\mu_k$.
		\item[] {\bf End do}
    	\end{algorithmic}
\end{algorithm}

When moving from the single-objective SSG setting to the multi-objective framework considered in this paper, the main modification to the SMG method is that each non-smooth stochastic objective~$h_i(x,\xi)$ is replaced by a smoothed approximation~$\tilde h_i(x,\xi,\mu)$ before forming the stochastic multi-gradient step. Accordingly, the expected objective~$f_i(x)=\mathbb{E}[h_i(x,\xi)]$ is replaced by its smoothed counterpart~$\tilde f_i(x,\mu)=\mathbb{E}[\tilde h_i(x,\xi,\mu)]$, and instead of using the original stochastic gradient estimates~$g_i(x_k,\xi_k)$, one uses the smoothed stochastic gradients
\[
\tilde g_i(x_k,\xi_k,\mu_k)=\nabla_x \tilde h_i(x_k,\xi_k,\mu_k), \qquad i=1,\ldots,m.
\]
The usual SMG procedure is then applied to these smoothed gradient estimates, that is, the stochastic multi-gradient is computed as
\[
\tilde g(x_k,\xi_k,\mu_k)=\sum_{i=1}^m \lambda_k^i \tilde g_i(x_k,\xi_k,\mu_k),
\]
where~$\lambda_k$ solves the same simplex-constrained minimum-norm problem as in the original SMG method. 

\section{Pareto-front stochastic multi-gradient (PF-SMG) method}
In the numerical tests of this paper, to compute good approximations of the entire Pareto front in a single run, we used the Pareto Front SMG (PF-SMG) algorithm, developed in~\cite{SLiu_LNVicente_2021}. The key idea of PF-SMG is to maintain and iteratively refine a list of candidate points that approximates the Pareto front. At each outer iteration, the algorithm first expands the current candidate list by adding perturbed points around each candidate point. Then, the algorithm proceeds by performing multiple short SMG runs from each candidate point and collecting the resulting endpoints. After that, it prunes the list by removing all dominated points. By repeating the procedure described above, PF-SMG yields a set of non-dominated points that approximates the Pareto front.

\end{document}